\documentclass[11pt]{amsart}
\usepackage[utf8]{inputenc}
\usepackage{mypreamble}

\newcommand{\be}{{\mathbf e}}

\renewcommand{\emptyset}{\varnothing}
\renewcommand{\min}{\operatorname{min}}
\title{The universal valuation of Coxeter matroids}

\author{Christopher Eur, Mario Sanchez, Mariel Supina}

\address{Stanford University. Stanford, CA. USA}
\email{chriseur@stanford.edu}

\address{University of California, Berkeley. Berkeley, CA. USA}
\email{mario\_sanchez@berkeley.edu}

\address{University of California, Berkeley. Berkeley, CA. USA}
\email{supina@math.berkeley.edu}

\begin{document}

\maketitle


\begin{abstract}
Coxeter matroids generalize matroids just as flag varieties of Lie groups generalize Grassmannians.
Valuations of Coxeter matroids are functions that behave well with respect to subdivisions of a Coxeter matroid into smaller ones.
We compute the universal valuative invariant of Coxeter matroids.  A key ingredient is the family of Coxeter Schubert matroids, which correspond to the Bruhat cells of flag varieties.
In the process, we compute the universal valuation of generalized Coxeter permutohedra, a larger family of polyhedra that model Coxeter analogues of combinatorial objects such as matroids, clusters, and posets.
\end{abstract}

\section{Introduction}

Let $V$ be a real finite dimensional vector space.  For a polyhedron $P \subseteq V$, let $\one_P: V \to \ZZ$ be the indicator function defined by $\one_P(x) = 1$ if $x\in P$ and 0 otherwise.  For a family $\mathscr P$ of polyhedra in $V$, let its \textbf{indicator group} $\mathbb I(\mathscr P)$ be the $\ZZ$-submodule of $\ZZ^{V}$ defined by
\[
\mathbb I(\mathscr P) :=\left\{ \sum_{P\in \mathscr P} a_P \one_P \in \ZZ^{V} \ \middle | \ a_P \in \ZZ,\ \textnormal{finitely many }a_P\textnormal{'s are nonzero}\right\}.
\]

\begin{definition}
A function $f: \mathscr P \to A$ from a family of polyhedra $\mathscr P$ to an abelian group $A$ is \textbf{valuative}\footnote{There are several variants of valuative functions in the literature.  Valuative functions as we defined here are sometimes called \emph{strongly valuative} functions.}
if there is a $\ZZ$-linear map $\widetilde f: \mathbb I(\mathscr P) \to A$ such that
$\widetilde f(\one_P) = f(P)$ for all $P\in \mathscr P$, or equivalently, if
\[
\sum_{i=1}^k a_i f(P_i) = 0 \quad\textnormal{whenever}\quad \sum_{i=1}^k a_i \one_{P_i} = 0 \quad\textnormal{for}\quad a_1, \ldots, a_k\in \ZZ \textnormal{ and } P_1, \ldots, P_k\in \mathscr P.
\]
We say that $f$ is a \textbf{valuation} on $\mathscr P$ in this case, and abuse the notation by writing $f$ also for $\widetilde f$.
\end{definition}

Valuative functions are fundamental objects in convex geometry, where they function as analogues of measures on convex bodies; see \cite{McM93} for a survey.  For (lattice) polyhedra, valuations provide a bridge between geometry and combinatorics particularly in the context of Ehrhart theory; see \cite{Joc19} for a survey.  The interaction between geometry and combinatorics is further strengthened in the study of valuations of (extended) generalized permutohedra.

\medskip
The set of \emph{(extended) generalized permutohedra} is a family of polyhedra that model several combinatorial objects including matroids, clusters, and posets \cite{Pos09, AA17}.  Valuative functions of these polyhedra were studied in \cite{AFR10, DF10} as invariants that behave well with respect to subdividing the associated combinatorial objects into smaller ones.  In particular, matroid subdivisions have rich connection to geometry, such as compactifications of fine Schubert cells \cite{Kap93, Laf99, Laf03}, tropical geometry \cite{Spe08, BEZ20}, and the $K$-theory of Grassmannians \cite{Spe09, FS12}.  The valuativeness of many matroid invariants, such as the beta invariant and the Tutte polynomial, is a witness to such geometric connections.  Most notably, the Schubert decomposition of Grassmannians appears in the \emph{$\mathcal G$-invariant} of matroids, an invariant that the authors of \cite{DF10} establish as the universal valuative invariant of matroids.

\medskip
We study valuative functions of the set of \emph{(extended) generalized Coxeter permutohedra}.
These polyhedra originate in Coxeter combinatorics, which recognizes that combinatorial objects associated to generalized permutohedra are inherently related to permutation groups, and accordingly studies their Coxeter analogues in arbitrary reflection groups.  Examples of such Coxeter analogues include {Coxeter matroids} \cite{GS87b,BGW03}, clusters \cite{HLT11}, and posets \cite{Rei92}.
The study of generalized Coxeter permutohedra as polyhedral models of these Coxeter combinatorial objects was initiated in \cite{ACEP20}.   Focusing on Coxeter matroids, we define the $\mathcal G$-invariant for {Coxeter matroids}, and show that it is the universal valuative invariant.  A key ingredient is the family of Coxeter matroids that correspond to Bruhat cells of flag varieties.

\subsection{Main results} \

\smallskip
Let $W$ be a finite reflection group with the associated root system $\Phi = (V,R)$ consisting of roots $R$ in a vector space $V$.  Let $\Pi_\Phi \subset V$ be a \textbf{$\Phi$-permutohedron}, which is the convex hull of the $W$-orbit of a general point in $V$.  Its normal fan is the Coxeter complex $\Sigma_\Phi$.

\begin{definition}\cite[Definition 4.3]{ACEP20}\label{def:GP}
An \textbf{extended generalized $\Phi$-permutohedron} is an extended deformation of $\Pi_\Phi$, that is, a (possibly unbounded) polyhedron $P \subseteq V$ such that each cone of its normal fan $\Sigma_P$ is a union of cones of $\Sigma_\Phi$.  Denote by $\mathsf{GP}_\Phi^+$ the set of all extended generalized $\Phi$-permutohedra.
\end{definition}

Our first main theorem explicitly describes the universal valuative function of $\mathsf{GP}^+_{\Phi}$ in terms of the \emph{Coxeter root cones} and \emph{tight containments}.  
Let $\Sigma_\Phi^\vee = \{\sigma^\vee \mid \sigma \in \Sigma_\Phi\}$ be the set of Coxeter root cones, which are the dual cones of cones in $\Sigma_\Phi$, and let $\tran(\Sigma_\Phi^\vee) = \{ C +v \mid C\in \Sigma_\Phi^\vee,\ v\in V\}$ be the set of all affine Coxeter root cones, which are translates of cones in $\Sigma_\Phi^\vee$.  An affine Coxeter root cone $C+v$ is said to \textbf{tightly contain} a polyhedron $P\subseteq V$ if
\[
C+v \supseteq P \quad\text{and}\quad  (\lineal(C)+v) \cap P \neq \emptyset
\]
(see \Cref{fig:tight containment} for an illustration).

\begin{maintheorem}\label{mainthm:Finvar}
Write $\{\be_{C+v}\}_{C+v \in \tran(\Sigma_\Phi^\vee)}$ for the standard basis of $\ZZ^{\oplus \tran(\Sigma_\Phi^\vee)}$.  The function
\[
\mathcal F: \mathsf{GP}^+_\Phi \to \ZZ^{\oplus \tran(\Sigma_\Phi^\vee)} \quad\text{defined by}\quad P \mapsto \sum_{\substack{C+v \in \tran(\Sigma_\Phi^\vee)\\ \textnormal{ tightly containing }P}}\be_{C+v}
\]
is the universal valuative function in the sense that for any valuative function $f: \mathsf{GP}_\Phi^+ \to A$ to an abelian group $A$, there exists a unique map $\varphi: \ZZ^{\oplus \tran(\Sigma_\Phi^\vee)} \to A$ such that $\varphi \circ \mathcal F = f$.
\end{maintheorem}

We prove \Cref{mainthm:Finvar} by establishing a more general statement \Cref{thm:tangent cone basis j}, which provides the universal valuative function for extended deformations of an arbitrary polyhedron.  Here, a useful tool we develop is \Cref{thm: tight containment is val}, which states that tight containments define valuations on extended deformations.
For a combinatorial application of \Cref{mainthm:Finvar} when $\Phi$ is a type $A$ root system, see \cite{AS20}, whose authors show further that $\mathcal{F}$ is a Hopf monoid morphism, and consequently construct new valuations on matroid polytopes, poset cones, and nestohedra.


\medskip
A feature unique to $\mathsf{GP}^+_\Phi$ not shared by extended deformations of an arbitrary polyhedron is the action of the group $W$ on $\mathsf{GP}^+_\Phi$, induced by the action of $W$ on $V$.  A valuative function $f$ on $\mathsf{GP}_{\Phi}^+$ is a \textbf{valuative invariant} if $f(w\cdot P) = f(P)$ for every $w\in W$ and $P \in \mathsf{GP}_\Phi^+$.  The universal valuative invariant of $\mathsf{GP}_{\Phi}^+$ is derived from \Cref{mainthm:Finvar} in \Cref{cor:G+invar}.

\medskip
We then turn to Coxeter matroids, which form a distinguished $W$-invariant subfamily of $\mathsf{GP}^+_\Phi$.  For $I$ a subset of a set of simple roots of $\Phi$, let $W_I$ be the corresponding parabolic subgroup of $W$ generated by the reflections corresponding to $I$.  Write $\leq$ for the Bruhat order on $W/W_I$, and for $w\in W$ and $B,B'\in W/W_I$, write $B\leq^w B'$ to mean $w^{-1}B \leq w^{-1} B'$.

\begin{definition}\label{def:Coxetermatroid}\cite[6.1.1]{BGW03}
For $I$ a subset of a set of simple roots of $\Phi$, a \textbf{Coxeter matroid of type $(\Phi,I)$}, or a \textbf{$(\Phi,I)$-matroid}, is a subset $M \subseteq W/W_I$ such that for every $w\in W$ there exists a unique $\leq^w$-minimal element in $M$.  The $\leq^w$-minimal element of $M$ is denoted $\min^w(M)$.
\end{definition}

By setting $W$ to be the permutation group and $W_I$ a maximal proper parabolic subgroup, one recovers the usual notion of matroids defined by Whitney \cite{Whi35} and independently by Nakasawa \cite{Nak35}.  We will say ``ordinary matroids'' to distinguish this usual notion of matroids from Coxeter matroids.  Like ordinary matroids, Coxeter matroids have several descriptions arising from different perspectives \cite{GS87a, BGW03}:
\begin{itemize}
    \item (Coxeter theory) \Cref{def:Coxetermatroid} of Coxeter matroids in terms of Bruhat order generalizes the characterization of ordinary matroids by the greedy algorithm optimization.
    \item (Lie theory) If $\Phi$ is crystallographic, (realizations of) Coxeter matroids correspond to points on the flag varieties of $\Phi$, just as (realizations of) ordinary matroids correspond to points on Grassmannians.
    \item (Polyhedral geometry) Coxeter matroids admit polyhedral models that are characterized in terms of edges being parallel to the roots of $\Phi$ (see \Cref{thm:GGMSCoxeter}).  This generalizes the theorem of \cite{GGMS87} which characterized the base polytopes of ordinary matroids in terms of edges being parallel to $\be_i - \be_j\in \RR^n$ for some $i\neq j \in \{1,\ldots, n\}$.
\end{itemize}
The last point above implies that the polyhedral models of Coxeter matroids of type $(\Phi,I)$ form a subfamily of $\mathsf{GP}_\Phi^+$, which we denote by $\mathsf{Mat}_{\Phi,I}$.  Our second main theorem computes the universal valuative invariant of Coxeter matroids.

\begin{maintheorem}\label{mainthm:Ginvar}
Write $\{U_B\}_{B\in W/W_I}$ for the standard basis of $\QQ^{W/W_I}$.  The function
\[
\mathcal G: \mathsf{Mat}_{\Phi,I} \to \QQ^{W/W_I} \quad\text{defined by}\quad M \mapsto \sum_{w\in W} U_{w^{-1}\cdot \min^w(M)}
\]
is the universal valuative invariant in the sense that for any valuative invariant $g: \mathsf{Mat}_{\Phi,I} \to A$ to a $\QQ$-vector space $A$, there exists a unique linear map $\psi: \QQ^{W/W_I} \to A$ such that $\psi \circ \mathcal G = g$.
\end{maintheorem}

We prove \Cref{mainthm:Ginvar} by establishing properties of \emph{Coxeter Schubert matroids} (\Cref{def:schubert}), which correspond to Bruhat cells of flag varieties (\Cref{rem:geomrem1}).  As a corollary, we establish that Coxeter Schubert matroids form a basis for the indicator space of isomorphism classes of Coxeter matroids  (\Cref{cor:Schubertsbasis}).  This parallels the fact that Bruhat cells give a basis of the cohomology ring of flag varieties.
As another application, for a class of Coxeter matroids known as \emph{delta-matroids}, we show that a well-studied invariant called the \emph{interlace polynomial} (\Cref{defn:interlace}) is a specialization of the $\mathcal G$-invariant and hence a valuative invariant (\Cref{thm:interlacePolynomial}).


\subsection{Relation to the work of Derksen and Fink}

For an ordinary matroid $M$ of rank $r$ on a ground set $[n]=\{1, \ldots, n\}$, Derksen and Fink \cite{DF10} define the $\mathcal G$-invariant of $M$ as follows:  For a permutation $w\in S_n$, let $X_i := \{w(1), \ldots, w(i)\}$ for $i = 1, \ldots, n$ and $X_0 := \emptyset$, and define $r_M(w) \in \{0,1\}^n$ by
\[
r_M(w)_i := \operatorname{rk}_M(X_i) - \operatorname{rk}_M(X_{i-1}) \quad {i = 1, \ldots, n},
\]
where $\operatorname{rk}_M$ is the rank function of $M$.  Let $\{U_\alpha \mid \alpha\in \{0,1\}^n\}$ be the standard basis of $\QQ^{\{0,1\}^n}$.  Then the $\mathcal G$-invariant of $M$ was defined in \cite{DF10} as
\[
\mathcal G(M) := \sum_{w\in S_n} U_{r(w)}.
\]
Let us relate this to our definition of $\mathcal G$-invariant in \Cref{mainthm:Ginvar}.  For $w\in S_n$, when the ground set $[n]$ is given weights $w^{-1}(i)$ for each $i\in [n]$, the greedy algorithm for matroids implies that the set $\{w(i) \in [n] \mid r(w)_i = 1\}$ is the minimal basis of $M$.  In other words, identifying elements of $\{0,1\}^n$ with subsets of $[n]$, we have $r(w) = w^{-1}\min^w(M)$, so that the $\mathcal G$-invariant in the sense of \cite{DF10} is exactly our $\mathcal G$-invariant
\[
\mathcal G(M) = \sum_{w\in S_n} U_{w^{-1}\min^w(M)}.
\]

While our work is thus a generalization of the work of Derksen and Fink, several arguments in \cite{DF10} fail fundamentally in the general Coxeter setting.  We highlight two differences between our work and \cite{DF10}:    
        \begin{itemize}
            \item \emph{Rank functions}: The rank function characterization of polymatroids and matroids plays a central role throughout \cite{DF10}.  While Coxeter submodularity was defined in \cite{ACEP20}, there is no known characterization of which $(\Phi,I)$-submodular functions are rank functions of $(\Phi,I)$-matroids.
            We circumvent this by introducing the notion of tight containment and translating certain results of \cite{DF10} into this geometric language.
            \item \emph{Schubert matroids}: A key feature of the type A root system utilized in \cite{DF10} is that Schubert matroid polytopes are the intersections of uniform matroid polytopes with the vertex cones of the standard permutohedron.  In other types however, these intersections are not necessarily Coxeter matroid polytopes, even in minuscule types (\Cref{rem:badintersection}).
            We circumvent this by borrowing tools from 0-Hecke algebras to establish properties of Coxeter Schubert matroids, along with a new argument for proving \Cref{mainthm:Ginvar}.
        \end{itemize}



\section{Valuative functions of extended deformations}\label{sec:valuation extended defs}

Since generalized Coxeter permutohedra are \emph{extended deformations} of certain polytopes, we first study valuative functions of extended deformations of an arbitrary polyhedron.
In \S\ref{subsec:deftight}, we introduce \emph{tight containments} as useful valuative functions of extended deformations.
In \S\ref{subsec:valdefbasis}, we compute a basis for the indicator group of the set of extended deformations of a polytope.
In \S\ref{subsec: universal val of extended defs}, we establish \Cref{thm:tangent cone basis j}, which includes \Cref{mainthm:Finvar} as a special case.

\smallskip
Throughout, let $V$ be a finite dimensional real vector space with an inner product $\langle \cdot, \cdot \rangle$.

\subsection{Extended deformations and tight containments}\label{subsec:deftight}

For a hyperplane $H = \{x\in V \mid \langle y,x \rangle = a \}$ defined by $y\in V$ and $a\in \RR$, we write $H^+ = \{x\in V \mid \langle y, x \rangle \geq a\}$ and $H^- = \{x\in V \mid \langle y,x \rangle \leq a\}$ for its two closed half-spaces.
A \textbf{polyhedron} $Q \subseteq V$ is a finite intersection of closed half-spaces, and is a \textbf{polytope} if it is compact.

\begin{figure}[h]
    \centering
    \includegraphics[width=.8\textwidth]{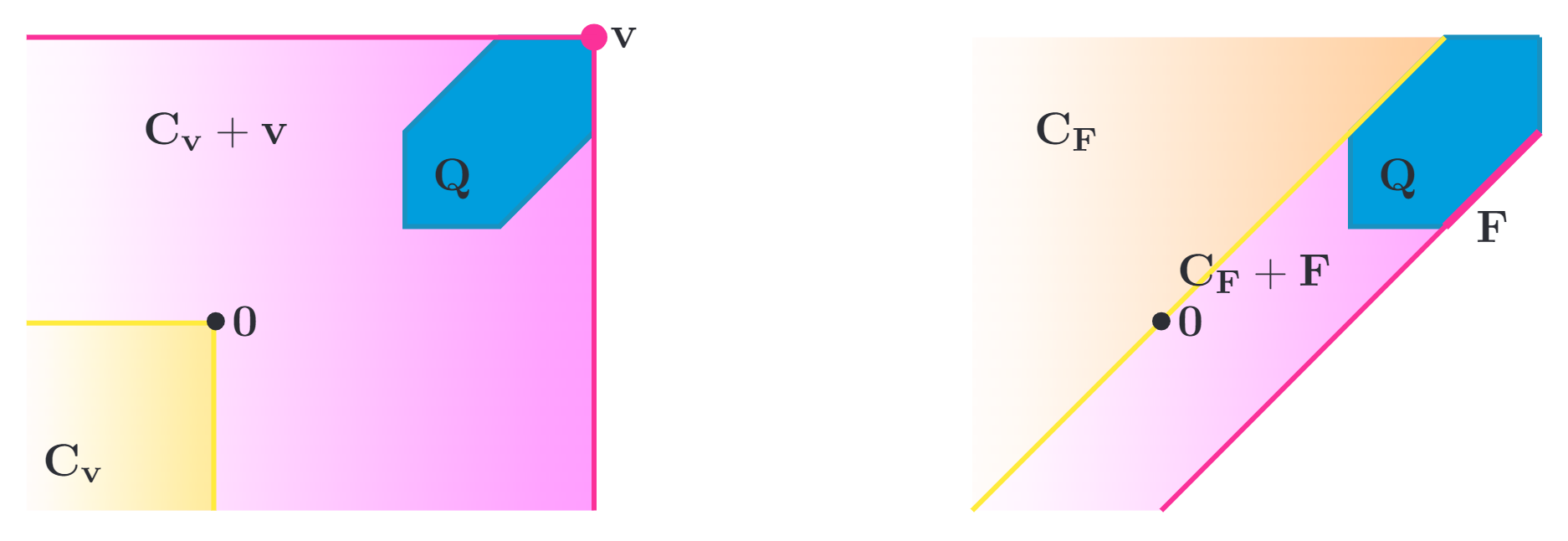}
    \caption{Tangent and affine tangent cones of two faces of the polygon $Q$.}
    \label{fig:tangent cones}
\end{figure}

\begin{notation}
For a polyhedron $Q\subseteq V$ and a face $F$ of $Q$, we set
    \begin{itemize}
        \item $\Sigma_Q$ to be the outer normal fan of $Q$,
        \item $\sigma_F$ to be the cone of $\Sigma_Q$ corresponding to the face $F$, which is the cone of linear functionals that attain their maximum value in $Q$ exactly on the face $F$,
        \item $C_F$ to be the \textbf{tangent cone} of $Q$ at $F$, which is the cone dual to $\sigma_F$, or explicitly, is $\operatorname{Cone}(v' - v \mid v \in F,\ v'\in Q)$,
        \item $C_F + F$ to be the \textbf{affine tangent cone} of $Q$ at $F$, which is the Minkowski sum of $C_F$ and $F$, or equivalently, is the translate $C_F + v$ of the tangent cone at $F$ by any $v\in F$,
        \item $\Sigma_Q^\vee := \{C_F \mid F \text{ is a face of } Q\} = \{\sigma^\vee \mid \sigma \in \Sigma_Q\}$ to be the set of tangent cones of $Q$, and
        \item $\tran(\Sigma_Q^\vee) := \{C+v \mid C\in \Sigma_Q^\vee, \ v\in V\}$ to be the set of all translates of tangent cones of $Q$.
    \end{itemize}
\end{notation}

\begin{definition}
Let $Q\subseteq V$ be a polyhedron.
A (possibly unbounded) polyhedron $P\subseteq V$ is an \textbf{extended deformation} of $Q$ if each cone of $\Sigma_{P}$ is a union of cones of $\Sigma_Q$.  The polyhedron $P$ is further a \textbf{deformation} of $Q$ if the supports of $\Sigma_{P}$ and $\Sigma_Q$ are equal, or equivalently, if $\Sigma_{P}$ is a coarsening of $\Sigma_Q$ (see \Cref{fig:deformations}).
\end{definition}

\begin{figure}[h]
    \centering
    \includegraphics[width=.6\textwidth]{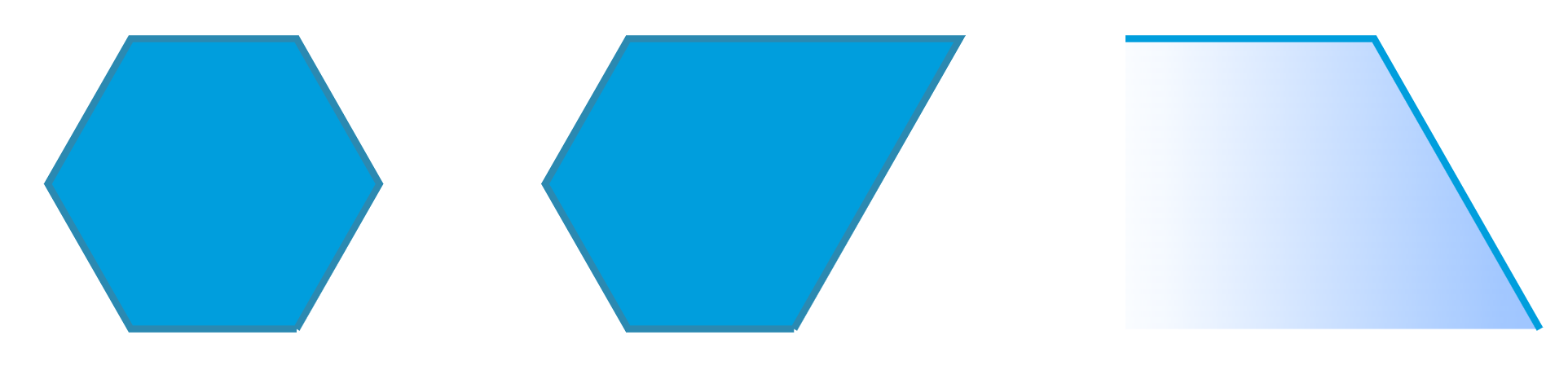}
    \caption{From left to right: A polytope $Q$, a deformation of $Q$, and an extended deformation of $Q$.}
    \label{fig:deformations}
\end{figure}

\begin{samepage}
We write
    \begin{itemize}
        \item $\Def^+(Q)$ for the set of all extended deformations of $Q$, and 
        \item $\Def(Q)$ for the set of all deformations of $Q$.
    \end{itemize}
\end{samepage}

We now introduce the notion of \emph{tight containments}, which will define useful valuative functions on the set of extended deformations.  
For a cone $C\subseteq V$, let $\lineal(C)$ be its \textbf{lineality space}, which is the maximal subspace of $V$ contained in $C$.

\begin{definition}\label{def:tight containment}
Let $C\subseteq V$ be a cone and $v\in V$.  We say that the translate $C+v$ of the cone $C$ \textbf{tightly contains} a polyhedron $P$ if $C+v$ contains $P$ and $P \cap (\lineal(C) + v) \not = \emptyset$.
\end{definition}

See \Cref{fig:tight containment} for illustrations of examples and non-examples of tight containments.

    \begin{figure}[h]
        \centering
        \begin{subfigure}{.2\linewidth}
            \centering
            \includegraphics[width=\linewidth]{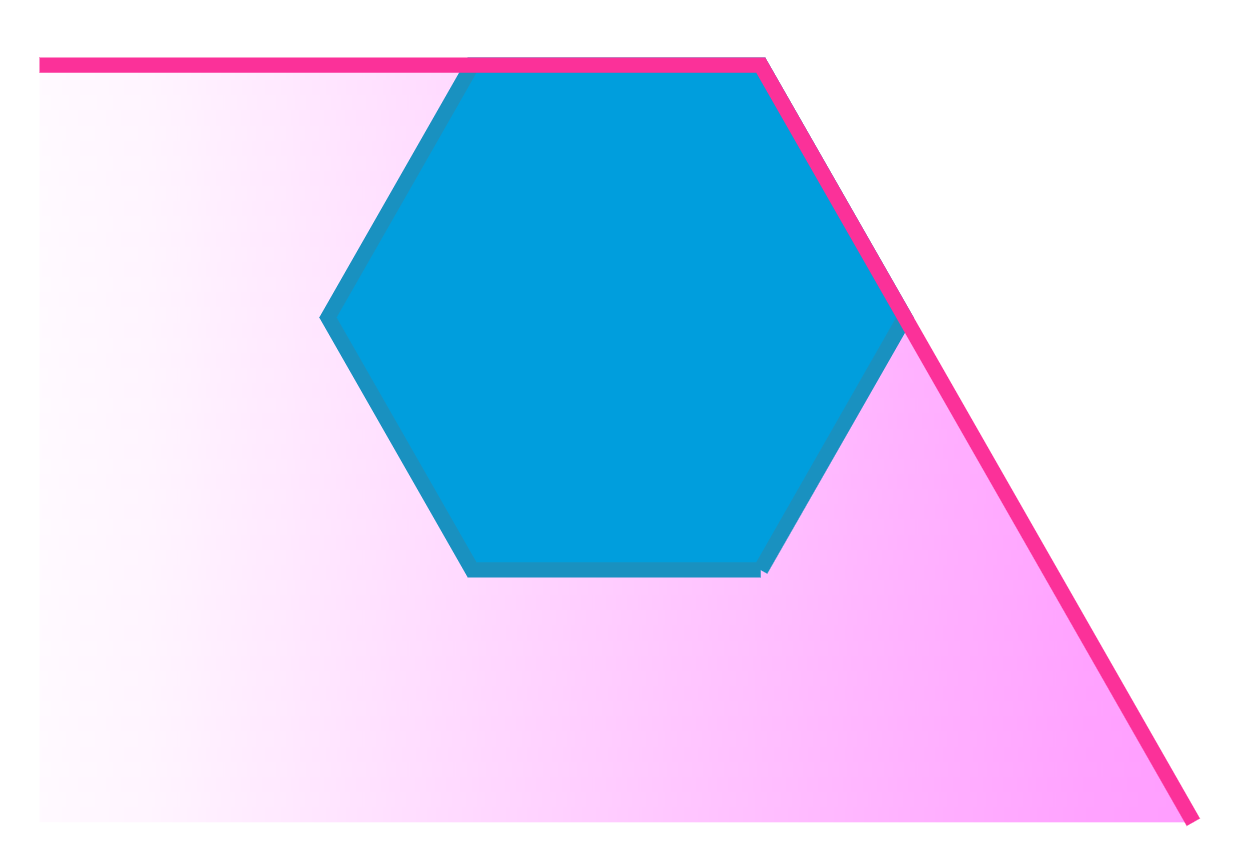}
            \caption{}
            \label{fig:tc1}
        \end{subfigure}
        \hspace{1mm}
        \begin{subfigure}{.2\linewidth}
            \centering
            \includegraphics[width=\linewidth]{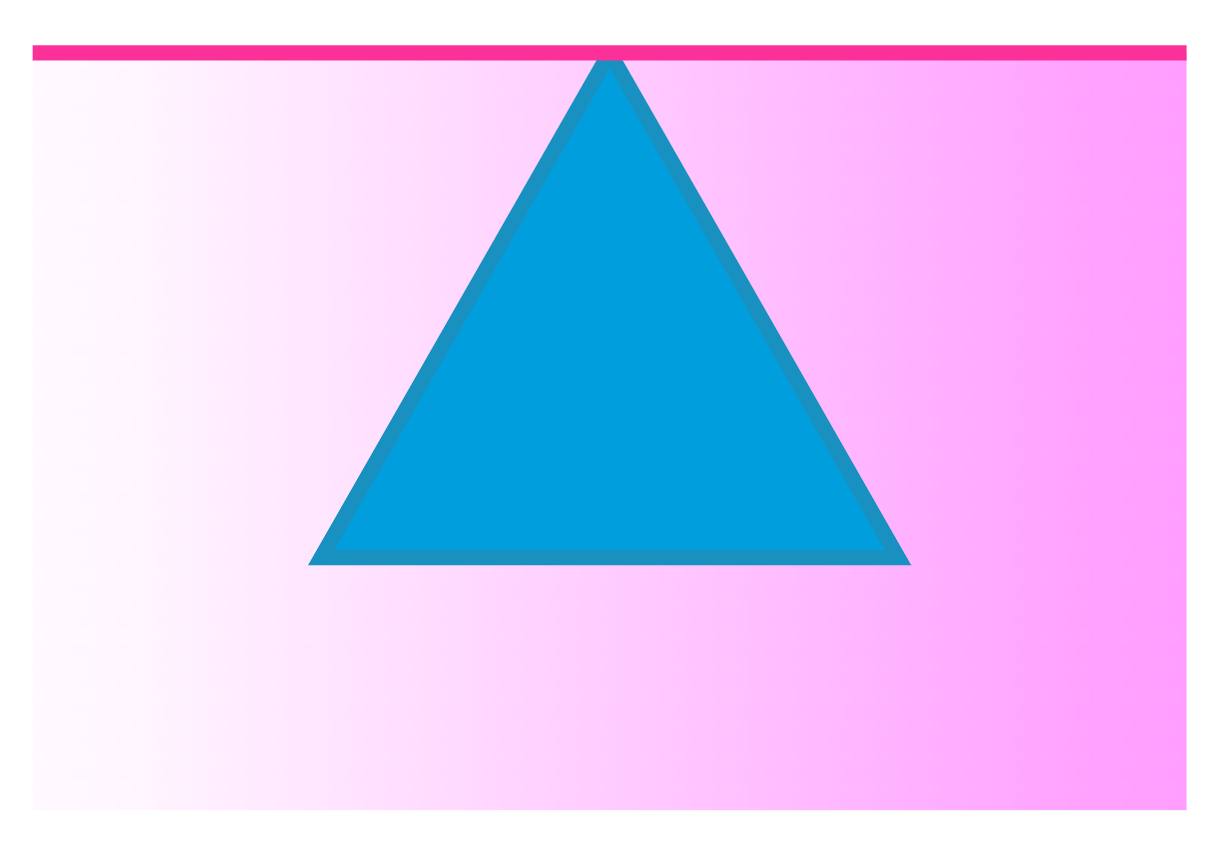}
            \caption{}
            \label{fig:tc2}
        \end{subfigure}
        \hspace{1mm}
        \begin{subfigure}{.2\linewidth}
            \centering
            \includegraphics[width=\linewidth]{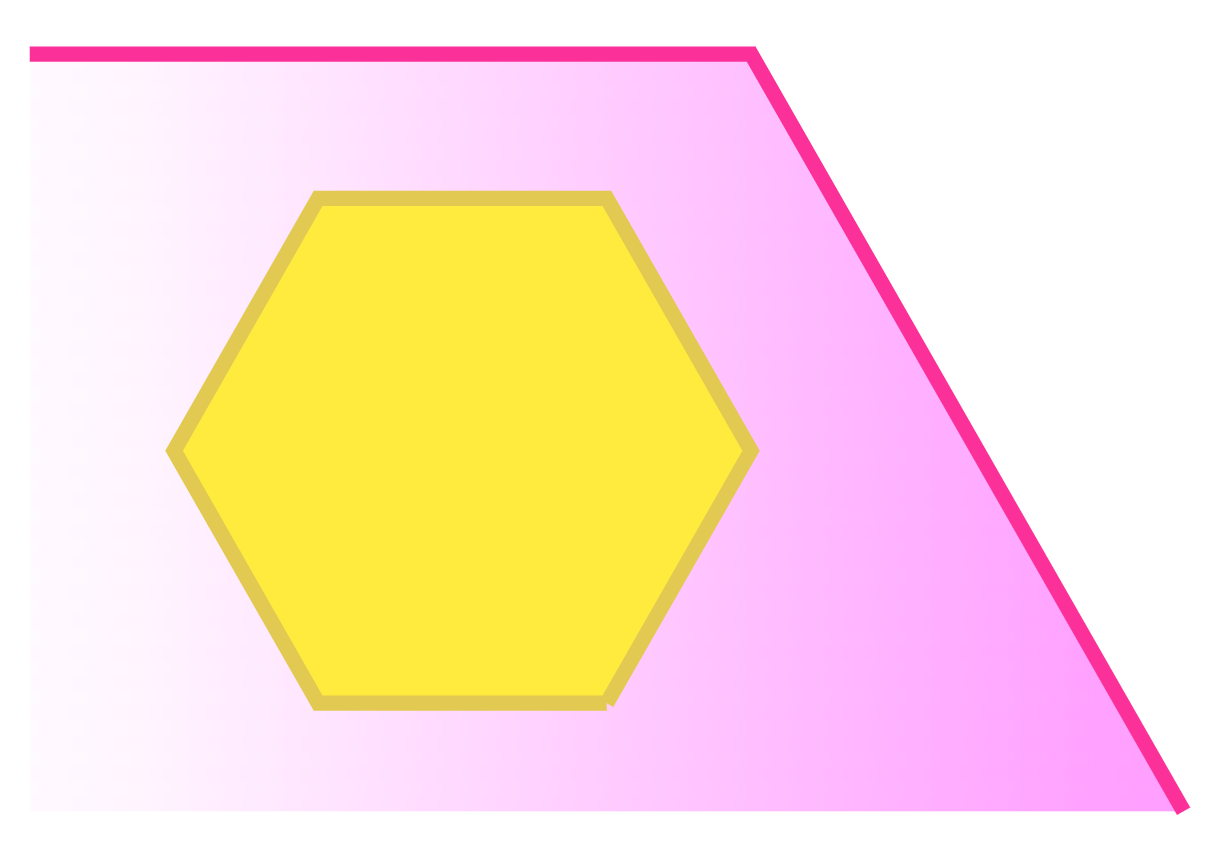}
            \caption{}
            \label{fig:tc3}
        \end{subfigure}
        \hspace{1mm}
        \begin{subfigure}{.2\linewidth}
            \centering
            \includegraphics[width=\linewidth]{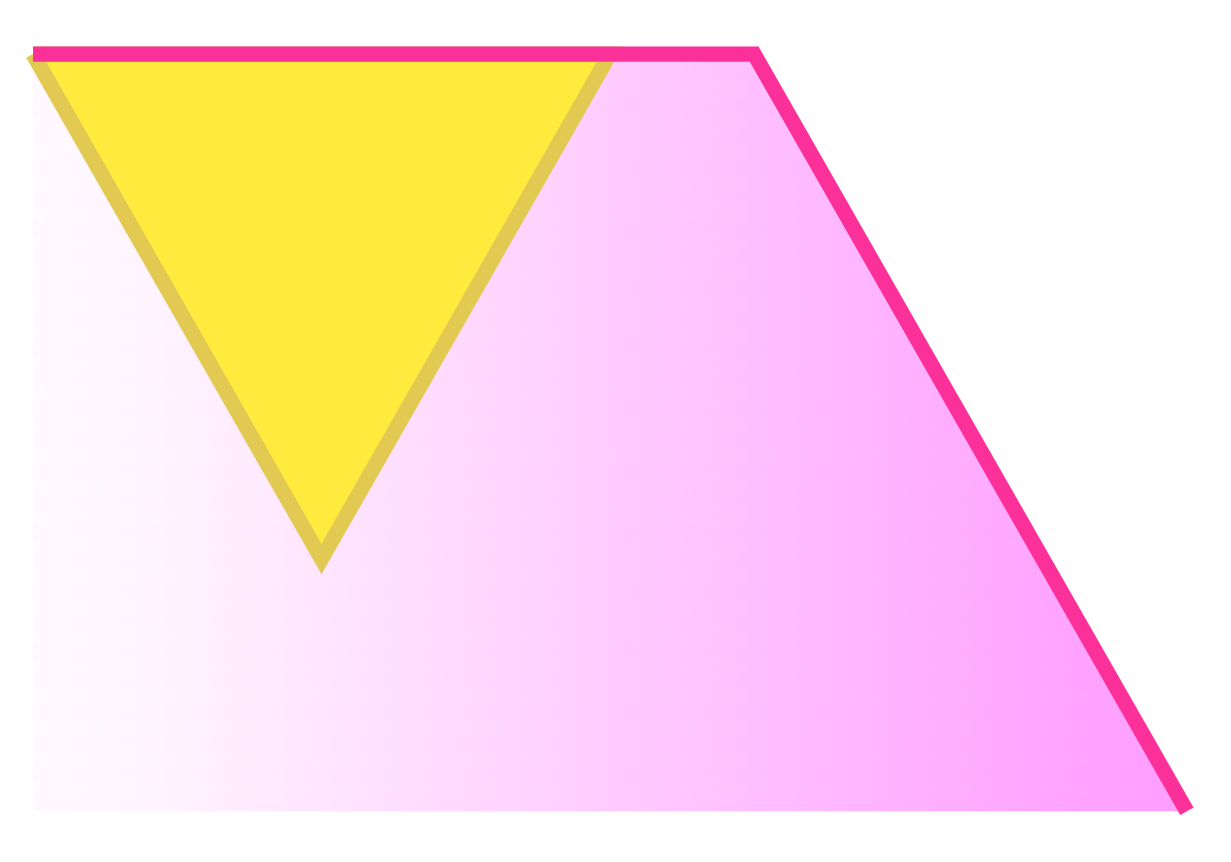}
            \caption{}
            \label{fig:tc4}
        \end{subfigure}
        \caption{Examples (\ref{fig:tc1}, \ref{fig:tc2}) and non-examples (\ref{fig:tc3}, \ref{fig:tc4}) of tight containments.}
        \label{fig:tight containment}
    \end{figure}

\begin{theorem}
\label{thm: tight containment is val}
    Let $F$ be a face of a polyhedron $Q\subset V$, and $v\in V$.
    Then the tight containment indicator function
        \[
        \tc{C_F+v}: \mathbb I(\Def^+(Q)) \to \ZZ \quad\text{defined by}\quad \tc{C_F+v}(P) :=
        \begin{cases}
        1 & \text{if $C_F+v$ tightly contains $P$} \\
        0 & \text{otherwise}
        \end{cases}
        \]
    is valuative.
\end{theorem}

Testing for tight containment in a cone does not in general define a valuative function for an arbitrary family of polyhedra.
For example, consider the cone $C = \operatorname{Cone}(\be_1, \be_1+\be_2) \subset \RR^2$ along with polyhedra $P_1 = \operatorname{Conv}(0, \be_2)$ and $P_2 = \operatorname{Conv}(0,-\be_2)$.  One has the relation $\one_{P_1}+\one_{P_2} = \one_{P_1\cup P_2} + \one_{P_1\cap P_2}$, but $0 = \tc{C}(P_1)+\tc{C}(P_2) \neq \tc{C}(P_1\cup P_2) + \tc{C}(P_1\cap P_2) = 1$.

\medskip
Our proof of \Cref{thm: tight containment is val} broadly consists of three steps: (i) tight containment for deformations can be expressed as a limit of containments in closed half-spaces, (ii) containment in a closed half-space is valuative, and (iii) taking a limit preserves valuativeness.
Let us begin by noting that the limit of a sequence of valuations with an elementary stabilization property is again a valuation.

\begin{lemma}\label{lem:limit of valuations}
Let $\mathscr P$ be a family of polyhedra in $V$, and $A$ an abelian group.
Suppose $(f_0, f_1, f_2, \ldots)$ is a sequence of valuations $\mathscr P \to A$ with the property that for any polyhedron $P\in\mathscr P$ there exists an integer $N_{P} \geq 0$ such that $f_{i}(P)=f_{N_{P}}(P)$ for all $i \geq N_P$.
    Then the function $f:\mathscr P\to A$ defined by $f(P):=f_{N_P}(P)$ is valuative.
\end{lemma}

\begin{proof}
Suppose $a_1, \ldots, a_k \in \ZZ$ and $P_1, \ldots, P_k \in \mathscr P$ satisfies $\sum_{i=1}^k a_i \one_{P_i} = 0$.
We need to show that $\sum_{i=1}^k a_i f(P_i) = 0$.  Writing $N = \max\{N_{P_i} \mid 1\leq i \leq k\}$, we have $\sum_{i=1}^k a_i f(P_i) = \sum_{i=1}^k a_i f_N(P_i)$ by definition of $f$, and $\sum_{i=1}^k a_i f_N(P_i) = 0$ since $f_N$ is valuative.
\end{proof}

\begin{notation}
We denote the function $f$ in \Cref{lem:limit of valuations} by $\lim_{i\to\infty} f_i$.
\end{notation}

For any family of polyhedra, containment in a closed half-space is a valuation.

\begin{lemma}\label{lem:half space containment}
Let $H^+\subset V$ be a closed affine half-space, and let $\mathscr P$ be a family of polyhedra in $V$.
The function $\inc{H^+}:\mathscr {P}\to\ZZ$ defined as
        $$
        \inc{H^+}(P) =
            \begin{cases}
            1 & \textnormal{if }P\subseteq H^+\\
            0 &\textnormal{otherwise}
            \end{cases}
        $$
is a valuation on $\mathscr {P}$.
\end{lemma}

\begin{proof}
If a function $f$ on the family of \emph{all} polyhedra is valuative, then its restriction $f|_{\mathscr P}$ is a valuation on $\mathscr P$.  We thus show that $\inc{H^+}$ is a valuative function on the family of all polyhedra in $V$.  For the family of all polyhedra, \cite[Propositions 3.2 \& 3.3]{McM09} states that a function $f$ is valuative if and only if it satisfies
\[
f(P) + f(P\cap L) = f(P \cap L^+) + f(P \cap L^-)
\]
for every polyhedron $P\subseteq V$ and hyperplane $L\subset V$.
Now, for an arbitrary polyhedron $P \subseteq V$ and a hyperplane $L\subset V$, the relation
\[
\inc{H^+}(P) + \inc{H^+}(P\cap L) = \inc{H^+}(P\cap L^+) + \inc{H^+}(P \cap L^-)
\]
follows from observing that:
\[
\begin{split}
P\cap L^+ \subseteq H^+ \textnormal{ \ or \ } P\cap L^- \subseteq H^+ &\implies P\cap L \subseteq H^+, \quad \textnormal{and}\\
P\cap L^+ \subseteq H^+ \textnormal{ and } P\cap L^- \subseteq H^+ &\implies P\subseteq H^+.
\end{split}
\]
The result follows.
\end{proof}

Lastly, we note the following alternate characterization of tight containment for extended deformations.
See \Cref{fig:Hyperplane condition} for an illustration.

    \begin{figure}[h]
        \centering
        \includegraphics[width=.7\textwidth]{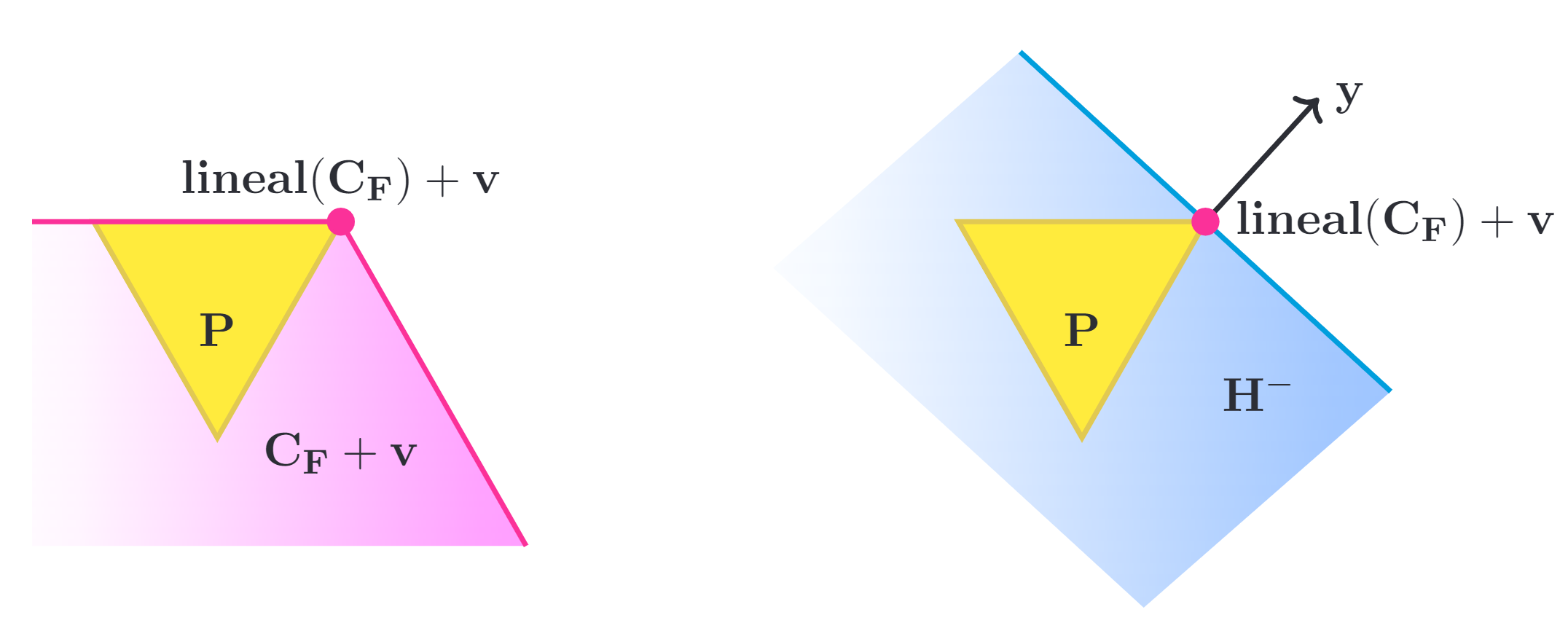}
        \caption{\small On the left, the polytope $P$ is tightly contained in the affine cone $C_F+v$.
        On the right, $P$ is contained in the affine half-space $H^-$ and intersects the boundary $H$ at $\lineal(C_F)+v$.
        (here the normal vector of $H$ must be chosen from the interior of the dual of $C_F$).
        \Cref{lem: tangent cone tight containment is half space containment} says that if $C_F$ is a tangent cone of $Q$ and $P\in\Def^+(Q)$, then these two conditions are equivalent.}
        \label{fig:Hyperplane condition}
    \end{figure}
    
\begin{lemma}\label{lem: tangent cone tight containment is half space containment}
Let $F$ be a face of a polyhedron $Q\subset V$, and $v\in V$.
Let $y\in\relint(\sigma_F)$ be a linear functional that is maximized on $Q$ at the face $F$, and let $H$ be the affine hyperplane $\{x\in V \mid \langle y,x\rangle = \langle y,v\rangle\}$. 
Then, for $P\in\Def^+(Q)$, the affine cone $C_F+v$ tightly contains $P$ if and only if $P\subseteq H^-$ and $\varnothing\subsetneq P \cap H \subseteq \lineal(C_F)+v$.
\end{lemma}
    
\begin{proof}
First, we note that $(C_F+v)\cap H = \lineal(C_F)+v$.
This is equivalent to the equality of subspaces $\lineal(C_F)=\{x\in C_F \mid \langle y,x\rangle=0\}$, which is clear since both subspaces can be written as $\Span\{w-w'\mid w,w'\in\ver(F)\}$.

Next, suppose that $P\in\Def^+(Q)$ is tightly contained in $C_F+v$.
Then $(\lineal(C_F)+v)\cap P$ is nonempty, so $H\cap P$ is also nonempty.
Furthermore, $P\subseteq C_F+v$ implies that $P\subseteq H^-$.
Finally, $P\cap H\subseteq (C_F+v)\cap H = \lineal(C_F) +v$.

For the converse, suppose that $P\subseteq H^-$ and $\emptyset\subsetneq P\cap H\subseteq\lineal(C_F) + v$.
Since $P\cap \lineal(C_F)+v\ne\emptyset$, it remains to show that $P\subseteq C_F+v$.
First, since $P\subseteq H^-$, the intersection $P\cap H$ is a face $F'$ of $P$, and $y$ is contained in the normal cone of $F'$.
As $P$ is a deformation of $Q$, the normal cone of $F'$ is a union of normal cones in $\Sigma_Q$ including $\sigma_F$.
Since the normal cone of $F'$ in $\Sigma_P$ contains the normal cone of $F$ in $\Sigma_Q$, the tangent cone $C_F$ of $Q$ at $F$ must contain the tangent cone $C_{F'}$ of $P$ at $F'$.
Let $u$ be a vertex of $F'$, so $u\in\lineal(C_F)+v$.
Since a polyhedron is always contained in all of its affine tangent cones, we have $P\subseteq C_{F'}+u\subseteq C_F+u = C_F+v$.
\end{proof}
    
We are now ready to prove \Cref{thm: tight containment is val}.
 
   \begin{proof}[Proof of \Cref{thm: tight containment is val}]
    Since the ``translation by $v$'' sending $P \in \Def^+(Q)$ to $\one_{P-v}$ is valuative, we can assume without loss of generality that $v=0$.
    Let $y\in\relint(\sigma_F)$ be a linear functional that is maximized on $Q$ at the face $F$, and define the hyperplane $H:=\{x\in V\mid\langle y,x\rangle = 0\}$ and the subspace $\ell:= \lineal(C_F)$.
    By Lemma \ref{lem: tangent cone tight containment is half space containment}, we can express the tight containment function $\tc{C_F}$ as
        $$
            \tc{C_F}(P) = \begin{cases}
            1, & P \subseteq H^-$ and $\emptyset \subsetneq P \cap H \subseteq \ell\\
            0, & \text{otherwise}.
            \end{cases}
        $$
    To show that this is a valuation, we will express $\tc{C_F}$ as a limit of valuations in the sense of \Cref{lem:limit of valuations}.
    We will construct a sequence of pairs of half-spaces $(H_i^-,{H_i'}^-)$ such that for large enough $i$, the polyhedra $P\in\Def^+(Q)$ contained in $H_i^-$ but not in ${H_i'}^-$ are exactly those with $\tc{C_F}(P) = 1$; in other words, $\tc{C_F}=\lim_{i\to\infty}(\inc{H_i^-}-\inc{H_i'^-})$.
    
    First, consider the case where $C_F$ itself is a half-space, that is, where the face $F$ has codimension $1$.
    Then the half-space $H^-$ is equal to $C_F$, and $\lineal(C_F)=H$.
    In this case, we can construct a sequence of affine half-spaces $\{H_i^-\}$ where each $H_i^-$ is contained in $H^-$ and $\lim_{i\to\infty} H_i^-=H^-$.
    The sequence of functions $\tc{C_F} = \lim_{i\to\infty} (\inc{H^-}-\inc{H_i^-})$ is a valuation.
    Here the functions $\inc{H^-}$ and $\inc{H_k^-}$ are the half-space containment functions as defined in Lemma \ref{lem:half space containment}.
    
    Now suppose that $C_F$ is not a half-space, and let $\{y_i\}_{i\in\NN}$ be a sequence of vectors in $\relint(\sigma_F)$ converging to $y$ that are not scalar multiples of $y$.
    For each $i$, define the hyperplane $H_i:=\{x\in V \mid \langle y_i, x\rangle = 0\}$.
    It is clear that this sequence of half-spaces $H_i^-$ converges pointwise to $H^-$ and that $\ell\subseteq H_i$ for all $i$.
    Let $d(X,Y)$ denote the minimum distance between two subsets $X,Y\subseteq V$, and construct a second sequence of affine hyperplanes denoted by $H_i'$ with the properties that ${H_i'}^- \subsetneq H_i^-$ and $\lim_{i \to \infty} d(H_i' \cap H, \ell) = 0$ (see \Cref{fig:H_i'}).
    The latter condition is well-defined because the normals $y_i$ were chosen not be scalar multiples of $y$ so that the intersections $H_i'\cap H$ are nonempty.
    These conditions on $H_i'$ guarantee that for any set $X \subset H$ that is bounded away from $\ell$, for sufficiently large $i$, if $H_i^-$ contains $X$ then so does ${H_i'}^-$.
    
    \begin{figure}
        \centering
        \resizebox{.6\linewidth}{!}{
\begin{tikzpicture}[scale=1.8,
  	point/.style={inner sep=1pt,circle,draw=black,fill=black,thick,anchor=base},
  	hyp/.style={color=mycyan,very thick},
  	affhyp/.style={color=mymagenta,very thick}]

\clip (-2,-1) rectangle (2.5,1.5);

\draw[very thick] (-4,0) to (4,0);

\draw[color=mymagenta!80!black, very thick, dashed] (-4,-3) to (4,1);
\draw[color=mycyan!80!black, very thick, dashed] (-4,-1.875) to (4,1.125);
\draw[color=myyellow!80!black, very thick, dashed] (-4,-1.125) to (4,.875);

\draw[color=mymagenta!80!black, very thick] (-4,-2) to (4,2);
\draw[color=mycyan!80!black, very thick] (-4,-1.5) to (4,1.5);
\draw[color=myyellow!80!black, very thick] (-4,-1) to (4,1);

\node[point] at (0,0) {};

\node[inner sep=2pt,circle,draw=mymagenta!80!black,fill=mymagenta!80!black,thick,anchor=base] at (2,0) {};
\node[inner sep=2pt,circle,draw=mycyan!80!black,fill=mycyan!80!black,thick,anchor=base] at (1,0) {};
\node[inner sep=2pt,circle,draw=myyellow!80!black,fill=myyellow!80!black,thick,anchor=base] at (.5,0) {};

\node[anchor=south] at (0,0) {$\boldsymbol{\ell}$};
\node[anchor=south] at (-1,0) {$\mathbf{H}$};

\end{tikzpicture}
}
        \caption{\small The first three elements in an example of the sequences of hyperplanes $\{H_i\}$ and $\{H_i'\}$, as described in the proof of \Cref{thm: tight containment is val}.
        The hyperplanes $H_i$ are drawn as solid lines, and the translations $H_i'$ are drawn as dashed lines.
        Observe that not only do the dashed hyperplanes $H_i'$ get closer and closer to their solid counterparts $H_i$, but the intersection points $H_i'\cap H$ also approach $\ell$, which is a stronger condition.}
        \label{fig:H_i'}
    \end{figure}
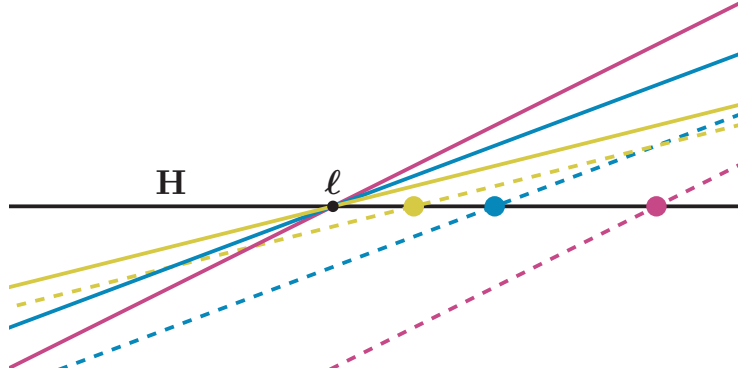
    
    We now show that the sequence of valuations $\{\inc{H_i^-} - \inc{{H_i'}^-}\}$ converges to $\tc{C_F}$ in the sense of \Cref{lem:limit of valuations}.
    We do this by partitioning $\Def^+(Q)$ into four subfamilies (\Cref{fig:tight containment proof}) and showing convergence for each subfamily.
    We then invoke Lemma \ref{lem:limit of valuations} to show that $\tc{C_F}$ is a valuation.
    Let $P\in\Def^+(Q)$.
        \begin{enumerate}[(I)]
            \item $P\subseteq H^-$ and $\emptyset\subsetneq P\cap H\subseteq \ell$.\\
            By Lemma \ref{lem: tangent cone tight containment is half space containment}, this case can be equivalently stated as the case where $P$ is tightly contained in $C_F$; hence we want to show that $\lim_{i\to\infty}(\inc{H_i^-}(P)-\inc{{H_i'}^-}(P))=1$.
            Since $y_i\in\relint(\sigma_F)$, it follows that $C_F$ (and hence $P$) is contained in $H_i^-$ for all $i$.
            Furthermore, since $P\cap \ell\ne\emptyset$, it holds that $P\not\subseteq {H_i'}^-$ for all $i$.
            Hence $\inc{H_i^-}(P)-\inc{{H_i'}^-}(P)=1$ for all $i$.
            
            \item $P\not\subseteq H^-$.\\
            Since the $H_i^-$ converge to $H^-$, for sufficiently large $i$ we will have that $P\not\subseteq H_i^-$.
            Since ${H_i'}^-\subset H_i^-$, we will also have $P\not\subseteq {H_i'}^-$.
            Thus $\lim_{i\to\infty}(\inc{H_i^-}(P)-\inc{{H_i'}^-}(P))=0$.
            
            \item $P\subseteq H^-$ and $P\cap H=\emptyset$.\\
            Since the $H_i^-$ converge to $H^-$, for sufficiently large $i$ we will have that $P\subseteq H_i^-$.
            Since $P$ is closed, $P$ is bounded away from $H$, and hence $P\subseteq {H_i'}^-$ for sufficiently large $i$.
            Thus $\lim_{i\to\infty}(\inc{H_i^-}(P)-\inc{{H_i'}^-}(P))=0$.
            
            \item $P\subseteq H^-$ and $\emptyset\subsetneq P\cap H\not\subseteq \ell$.\\
            If $P \not \subseteq H_i^-$ then we also have $P \not \subseteq {H_i'}^- \subset H_i^-$.
            Thus $\inc{H_i^-}(P) - \inc{{H_i'}^-}(P) = 0$.
            Suppose now that $P \subset H_i^-$.
            Since $P\in\Def^+(Q)$, all edge directions of $P$ must be edge directions of $Q$.
            Furthermore, the only edge directions of $Q$ that are contained in $H^-$ are the ones generating $C_F$, and so the only edge directions contained in $H$ must necessarily be the ones in $\lineal(C_F)=\ell$.
            Thus $P\cap H$ must be contained in an affine translation of $\ell$, and $P$ is tightly contained in the affine cone $(P\cap H) + C_F$.
            By construction, the collection of edge directions of $Q$ contained in $H^-$ is the same as the collection contained in $H_i^-$.
            Thus if ${H_i'}^-$ contains $P \cap H$, then it also contains $P$. 
            Since $P \cap H$ is contained in a translation of $\ell$, it is bounded away from $\ell$.
            As we saw in the construction of ${H_i'}^-$, this means that for sufficiently large $i$, the half-space ${H_i'}^-$ contains $P \cap H$ if $H_i^-$ contains it.
            This means that ${H_i'}^-$ contains $P$.
            Hence $\lim_{i\to\infty}(\inc{H_i^-}(P)-\inc{{H_i'}^-}(P))=0$.
        \end{enumerate}
        This concludes the proof of \Cref{thm: tight containment is val}.
    \end{proof}
        
        \begin{figure}[h]
        \centering
        \begin{subfigure}{.2\linewidth}
            \centering
            \includegraphics[width=\linewidth]{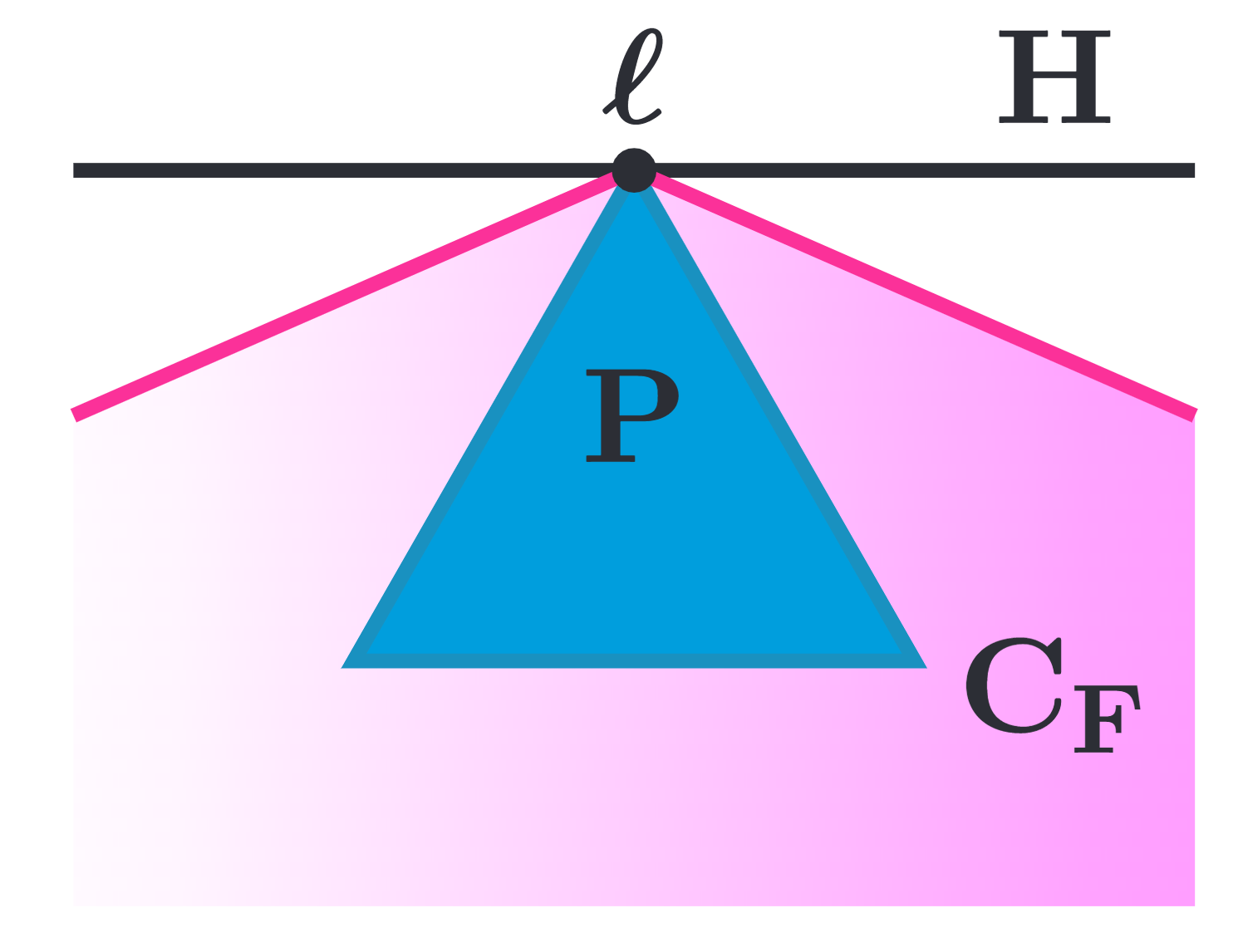}
            \caption{}
            \label{fig:case1}
        \end{subfigure}
        \hspace{1mm}
        \begin{subfigure}{.2\linewidth}
            \centering
            \includegraphics[width=\linewidth]{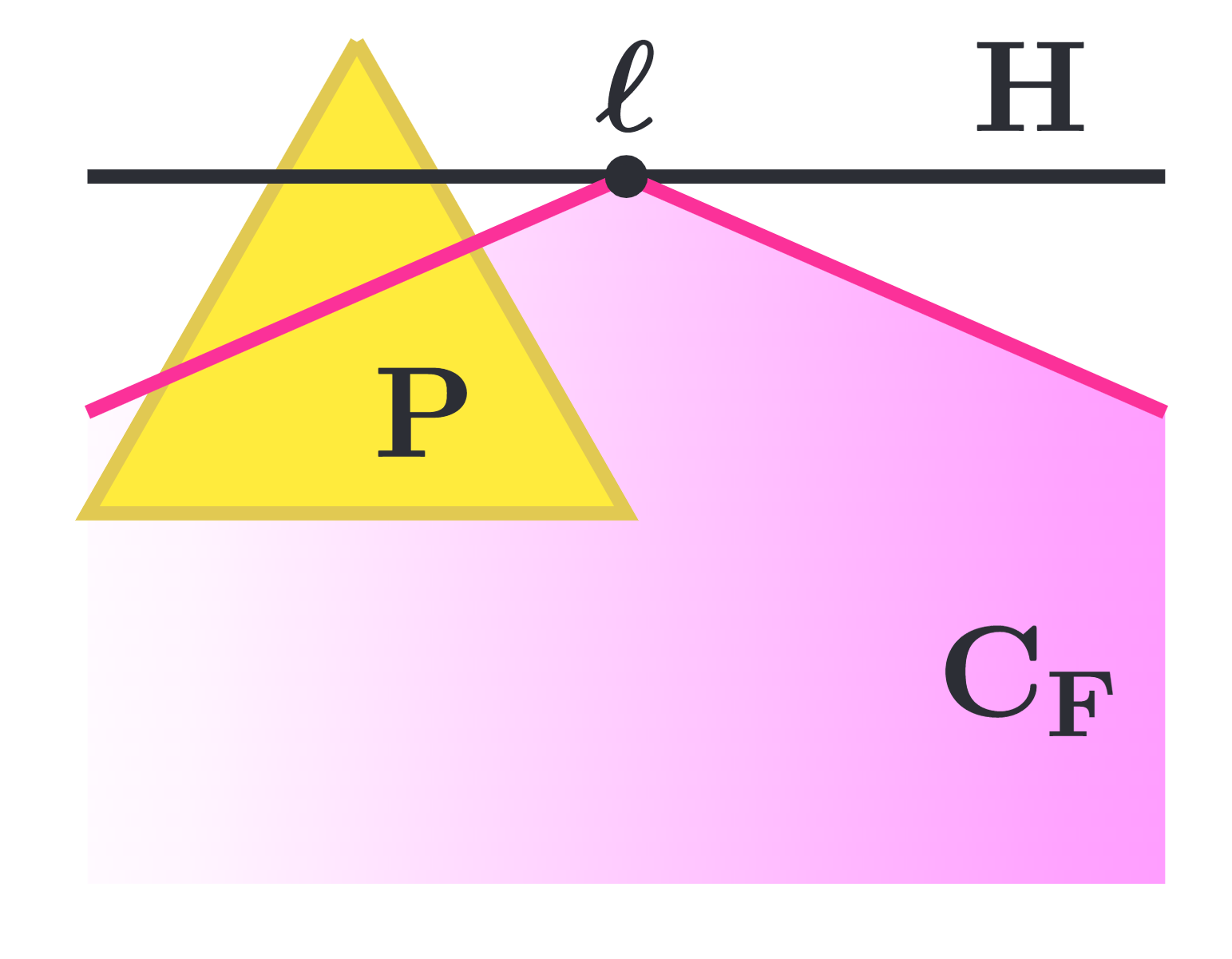}
            \caption{}
            \label{fig:case2}
        \end{subfigure}
        \hspace{1mm}
        \begin{subfigure}{.2\linewidth}
            \centering
            \includegraphics[width=\linewidth]{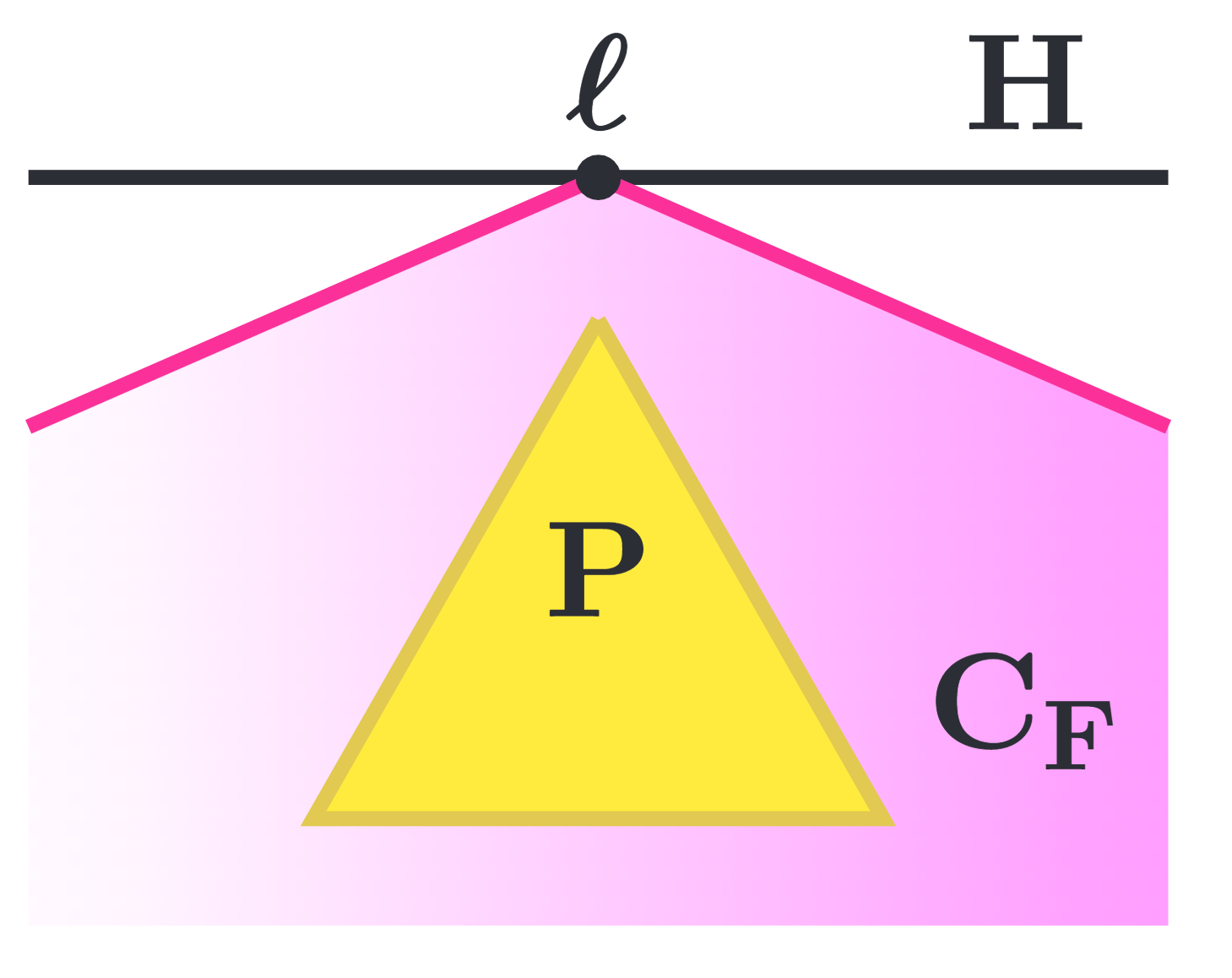}
            \caption{}
            \label{fig:case3}
        \end{subfigure}
        \hspace{1mm}
        \begin{subfigure}{.2\linewidth}
            \centering
            \includegraphics[width=\linewidth]{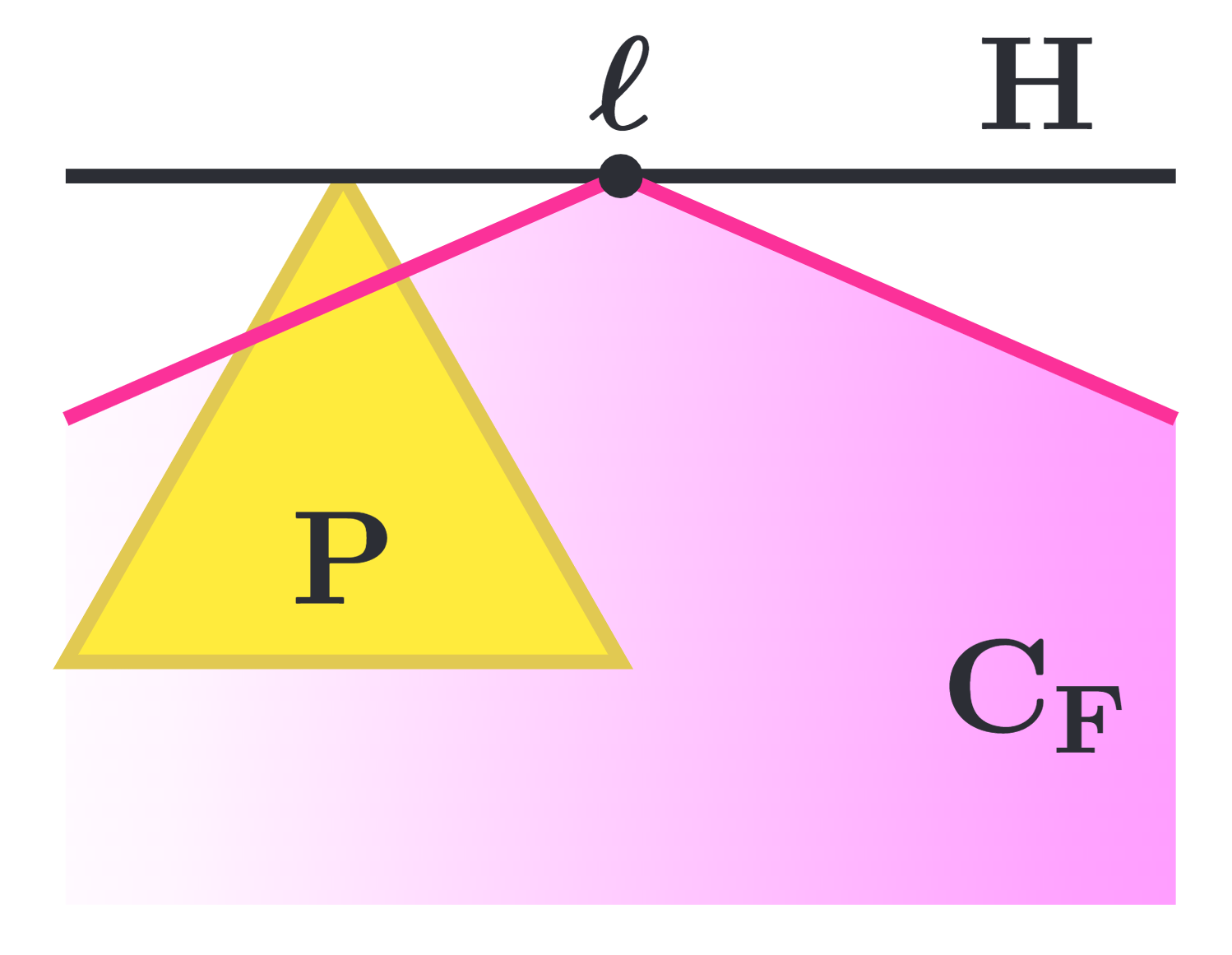}
            \caption{}
            \label{fig:case4}
        \end{subfigure}
        \caption{Example of each of the four cases from the proof of \Cref{thm: tight containment is val}.}
        \label{fig:tight containment proof}
    \end{figure}

\subsection{A basis for the indicator group of  extended deformations}\label{subsec:valdefbasis}

Recall that $\tran(\Sigma_Q^\vee) = \{C+v \mid C\in \Sigma_Q^\vee,\ v \in V\}$ denotes the set of all translates of tangent cones of a polyhedron $Q$.

\begin{theorem}\label{thm:tangent cone basis}
Let $Q\subset V$ be a polyhedron.  The collection $\{\one_{C+v} \mid C+v \in \tran(\Sigma_Q^\vee)\}$ of indicator functions of translations of tangent cones of $Q$ is a basis for the $\ZZ$-module $\ind(\Def^+(Q))$.
\end{theorem}

We begin by showing that the proposed collection spans $\ind(\Def^+(Q))$. Our proof closely mirrors one given in \cite[Theorem 4.2]{DF10}.
We prepare by recalling the Brianchon-Gram decomposition theorem.
When a polyhedron $Q$ has lineality space $L$, we say that a face $F$ of $Q$ is \textbf{relatively bounded} if $F/L$ is a bounded face of $Q/L$.

\begin{theorem}\cite{Bri37, Gra74}\label{thm:Brianchon Gram}
Let $Q$ be a polyhedron with lineality dimension $\ell$.
Then
        $$
        \one_Q = \sum_{F} (-1)^{\dim F-\ell}\one_{C_F + F}
        $$
where this sum is taken over all relatively bounded faces $F$ of $Q$.
\end{theorem}

\begin{proposition}\label{prop:tangent cone spanning}
Let $Q\subset V$ be a polyhedron.
The $\ZZ$-module $\ind(\Def^+(Q))$ is spanned by $\{\one_{C+v} \mid C+v \in \tran(\Sigma_Q^\vee)\}$, the indicator functions of translations of tangent cones of $Q$.
\end{proposition}
 
    \begin{proof}
    Let $P\in \Def^+(Q)$; i.e., $P$ is a polyhedron whose normal fan coarsens a subfan of $\Sigma_Q$.
    Given $\epsilon>0$, consider the Minkowski sum $P+\epsilon Q\subseteq\RR^n$.
    Since the normal fan of the Minkowski sum of two polyhedra is the common refinement of the two normal fans, we have that $\Sigma_{P+\epsilon Q}$ is a subfan of $\Sigma_Q$.
    Thus, the affine tangent cones of $P+\epsilon Q$ all have the form $C_F+v_{F,\epsilon}$ for some $C_F\in\Sigma_Q^\vee$ and vertex $v_{F,\epsilon}$ of $P+\epsilon Q$ where $F$ is a face of $Q$.
    By the Brianchon-Gram Theorem \ref{thm:Brianchon Gram}, we have that
        \begin{equation}\label{eq:Brianchon Gram proof}
            \one_{P+\epsilon Q} = \sum_{\substack{\text{Relatively}\\\text{bounded}\\\text{faces }{\hat F}\\\text{of } P+\epsilon Q}} (-1)^{\dim {\hat F}}\one_{C_F+v_{F,\epsilon}}
        \end{equation}
    where $C_F+v_{F,\epsilon}$ is the affine tangent cone of $P+\epsilon Q$ corresponding to the face ${\hat F}$.
    As $\epsilon$ goes to $0$, the left side of \Cref{eq:Brianchon Gram proof} converges pointwise to $\one_P$.
    Since vertices of $P+\epsilon Q$ converge to vertices of $P$ as $\epsilon\to0$, each affine cone $C_F+v_{F,\epsilon}$ on the right side of \Cref{eq:Brianchon Gram proof} converges to a tangent cone of $Q$ translated by a vertex of $P$.
    \end{proof}

We now complete the proof of \Cref{thm:tangent cone basis} by using tight containments to establish the linear independence of the proposed basis for $\mathbb I(\Def^+(Q))$.

\begin{proof}[Proof of \Cref{thm:tangent cone basis}]
\Cref{prop:tangent cone spanning} states that $\{\one_{C+v} \mid C+v \in \tran(\Sigma_Q^\vee)\}$ spans $\ind(\Def^+(Q))$, so it remains to show that these indicator functions are linearly independent.
Let $\{C_i+v_i\mid 1\le i\le k\}$ be a finite collection of translates of cones in $\Sigma_Q^\vee$, and suppose that there exist $a_i\in \ZZ$ such that  
    \begin{equation}\label{eq:linear dep}
        \sum_{i=1}^k a_i\cdot\one_{C_i+v_i} = 0.
    \end{equation}
There exists some $i$ so that $C_i+v_i$ contains no other $C_j+v_j$ for $j\neq i$.  Suppose without loss of generality that $i=1$.
We then have
    \begin{align*}
        \tc{C_1+v_1}({C_i+v_i})
        &:= \begin{cases}
        1, &C_1+v_1 \text{ tightly contains } C_i+v_i\\
        0, &\text{otherwise}
        \end{cases}\\
        &=\begin{cases}
        1, &i=1\\
        0, &i\ne 1.
        \end{cases}
    \end{align*}
Since $\tc{C_1+v_1}:\Def^+(P)\to\ZZ$ is a valuation by \Cref{thm: tight containment is val}, we can apply $\tc{C_1+v_1}$ to both sides of the equation \eqref{eq:linear dep} to obtain
    \begin{align*}
        0 = \tc{C_1+v_1}\Big(\sum_{i=1}^k a_i\cdot\one_{C_i+v_i}\Big)
        = \sum_{i=1}^k a_i\cdot\tc{C_1+v_1}(\one_{C_i+v_i})
        = a_1.
    \end{align*}
By repeating this process, we conclude that $a_i=0$ for all $i=1, \ldots, k$.
Thus the functions $\{\one_{C+v} \mid C+v \in \tran(\Sigma_Q^\vee)\}$ are linearly independent.
\end{proof}

\subsection{The universal valuative function of extended deformations}\label{subsec: universal val of extended defs}
We now prove \Cref{mainthm:Finvar} by establishing a more general statement for extended deformations of an arbitrary polyhedron.
For a set $S$, denote by $\{\be_i \mid i \in S\}$ the standard basis of $\ZZ^{\oplus S}$.

\begin{theorem}\label{thm:tangent cone basis j}
Let $Q \subseteq V$ be a polyhedron.  The map
\[
{\mathcal F}: \mathbb I(\Def^+(Q)) \to \ZZ^{\oplus \tran(\Sigma_Q^\vee)} \quad \textnormal{determined by}\quad \one_P \mapsto \sum_{\substack{C+v \in \tran(\Sigma_Q^\vee)\\ \textnormal{ tightly containing }P}}\be_{C+v}
\]
is a well-defined $\ZZ$-linear isomorphism.  In other words, by abusing the notation as before to write $\mathcal F$ also for the map $\Def^+(Q) \to \ZZ^{\oplus \tran(\Sigma_Q^\vee)}$ defined by $P \mapsto \mathcal F(\one_P)$, we have that for any valuative function $f: \Def^+(Q) \to A$ to an abelian group $A$, there exists a unique linear map $g: \ZZ^{\oplus \tran(\Sigma_Q^\vee)} \to A$ such that $f = g\circ \mathcal F$.
\end{theorem}

We prepare the proof of \Cref{thm:tangent cone basis j} with an observation about tight containments of tangent cones.

\begin{lemma}\label{lem:tight containment for cones}
Let $Q\subseteq V$ be a polyhedron, and $C'+v' \in \tran(\Sigma_Q^\vee)$ a translate of a tangent cone of $Q$.  Another translate of a tangent cone $C+v \in \tran(\Sigma_Q^\vee)$ of $Q$ tightly contains $C'+v'$ if and only if $C+v$ is an affine tangent cone of $C'+v'$.
\end{lemma}

\begin{proof}
First consider the $v = v'$ case, which is equivalent to the case of $v = v' = 0$ by translating both $C+v$ and $C'+v'$ by $-v$.  For cones, tight containment is equivalent to containment, since every cone contains the origin in its lineality space.  Moreover, for tangent cones $C$ and $C'$ of $Q$, we have $C'\subseteq C$ if and only if ${C}^\vee$ is a face of ${C'}^\vee$ (since both dual cones are elements of the fan $\Sigma_Q$), or equivalently, if and only if $C$ is a tangent cone of $C'$.

For the general case, suppose $C$ tightly contains $C'+v'-v$, so that there exists $x \in \operatorname{lineal}(C) \cap (C'+v'-v)$.  Since $x\in \operatorname{lineal}(C)$ implies that at most one of $x+y$ and $x-y$ lie in $C$ for any nonzero $y\notin \operatorname{lineal}(C)$, the smallest face of $C'+v'-v$ containing $x$ is contained in $\operatorname{lineal}(C)$.  Thus, the affine span of this face contains the origin, and hence $v'-v\in \operatorname{lineal}(C)$.  Since the translate of $C$ by the element $v'-v$ in its lineality space is equal to $C$ itself, we are thus reduced to the case of $v = v'$.
\end{proof}

\begin{proof}[Proof of \Cref{thm:tangent cone basis j}]
By definition, the map $\mathcal F$ is defined by $\mathcal F(\one_P) = \sum_{C+v\in \tran(\Sigma_Q^\vee)} \tc{C+v}(P)\be_{C+v}$.  Since tight containments define valuative functions (\Cref{thm: tight containment is val}), the map $\mathcal F$ is $\ZZ$-linear if it is well-defined.  For $\mathcal F$ to be well-defined, the summation $\sum_{C+v\in \tran(\Sigma_Q^\vee)} \overline{j_{C+v}}(P)\be_{C+v}$ need be finite for any $P\in \Def^+(Q)$.  It suffices to check this for $P = C'+v' \in \tran(\Sigma_Q^\vee)$ since any $P\in \Def^+(Q)$ is a linear combination over finitely many translates of tangent cones of $Q$.  \Cref{lem:tight containment for cones} implies that
\begin{equation}\label{eq:tight containment for cones}\tag{\textdagger}
\mathcal F(\one_{C'+v'}) = \sum_{C+v} \be_{C+v}, 
\end{equation}
where the sum is over all affine tangent cones $C+v$ of $C'+v'$, which is a finite sum.

To establish that $\mathcal F$ is an isomorphism, we first note that
\Cref{thm:tangent cone basis} gives an isomorphism
\[
\varphi: \ZZ^{\oplus \tran(\Sigma_Q^\vee)} \overset\sim\to \mathbb I(\Def^+(Q)) \quad \textnormal{defined by}\quad \be_{C+v} \mapsto \one_{C+v}.
\]
Now consider the composition
\[
\mathcal F \circ \varphi: \ZZ^{\oplus \tran(\Sigma_Q^\vee)} \to \ZZ^{\oplus\tran(\Sigma_Q^\vee)}\quad \textnormal{defined by}\quad \be_{C'+v'} \mapsto \sum_{C+v\in \tran(\Sigma_Q^\vee)}\overline{j_{C+v}}(C'+v')\be_{C+v},
\]
We claim that for any finite subset $S\subset \tran(\Sigma_Q^\vee)$, there exists a finite subset $S' \subset \tran(\Sigma_Q^\vee)$ containing $S$ such that the restriction of $\mathcal F \circ \varphi$ has image $\ZZ^{S'}$ and is an isomorphism $\ZZ^{S'} \overset{\sim}\to \ZZ^{S'}$.  Given any finite $S\subset \tran(\Sigma_Q^\vee)$, let $S'$ be the set of affine tangent cones of the elements in $S$.  The equation \eqref{eq:tight containment for cones} established by \Cref{lem:tight containment for cones} implies that a linear order on $S'$ that refines the containment relation makes the matrix of the map $\mathcal F \circ \varphi: {\ZZ^{S'}} \to \ZZ^{S'}$ triangular with 1's on the diagonal.
\end{proof}

Recall from \Cref{def:GP} that an extended generalized $\Phi$-permutohedron is an extended deformation of a $\Phi$-permutohedron $\Pi_\Phi$, and
\[\mathsf{GP}_\Phi^+ := \Def^+(\Pi_\Phi), \quad \text{the set of all extended generalized $\Phi$-permutohedra.}\]

\begin{proof}[Proof of \Cref{mainthm:Finvar}]
Set $Q = \Pi_\Phi$ in \Cref{thm:tangent cone basis j}.
\end{proof}

\bigskip
The remainder of this section will only be used in the proof of \Cref{prop: existence}.  Nonetheless, because the results may be of independent interest in polyhedral geometry, we develop them here in the setting of deformations of arbitrary polytopes.
Let $\mathscr V$ be the set of vertices of a deformation of a polytope $Q\subset V$.  Let
\[
\Def(Q)|_{\mathscr V} := \{Q' \in \Def(Q) \mid Q' \subseteq Q,\ \operatorname{Vert}(Q') \subseteq \mathscr V\}
\]
be the set of deformations of $Q$ that have vertices in $\mathscr V$.  Let $\min(\Sigma_Q^\vee)$ be the set of minimal tangent cones of $Q$, that is, the tangent cones of the vertices of $Q$.  Denote by $X_{\mathscr V} := \{C+v \mid C \in \min(\Sigma_Q^\vee),\ v \in \mathscr V\} \subset \tran(\Sigma_Q^\vee)$, and define a function
\[
\mathcal F_{\mathscr V}: \mathbb I(\Def(Q)|_{\mathscr V}) \to \ZZ^{X_{\mathscr V}}
\quad\textnormal{by}\quad \one_{P} \mapsto \sum_{\substack{C+v \in X_{\mathscr V}\\ \textnormal{tightly containing }P}} \be_{C+v}.
\]

\begin{proposition}\label{prop:restrictedFinvar}
Let $Q$, $\mathscr V$, and $\mathcal F_{\mathscr V}$ be as above.  There is a $\ZZ$-linear injection $\mathcal E_{\mathscr V}: \ZZ^{X_{\mathscr V}} \to \ZZ^{\oplus \tran(\Sigma_Q^\vee)}$ such that the following diagram commutes:
\[
\begin{tikzcd}
\mathbb I(\Def(Q)|_{\mathscr V}) \arrow[r, "\mathcal F_{\mathscr V}"] \arrow[d, hook] & \ZZ^{X_{\mathscr V}} \arrow[d, "\mathcal E_{\mathscr V}"]\\
\mathbb I(\Def^+(Q)) \arrow[r, "\mathcal F"] & \ZZ^{\oplus\tran(\Sigma_Q^\vee)}.
\end{tikzcd}
\]
\end{proposition}

\begin{proof}
A translate of tangent cone of $Q$ tightly contains $P \in \Def(Q)|_{\mathscr V}$ only if it is of the form $C+v$ where $C\in \Sigma_Q^\vee$ and $v\in \mathscr V$.  Let us denote $Y_{\mathscr V} := \{C+ v \mid C\in \Sigma_Q^\vee, \ v\in \mathscr V\}$.  For each $C\in \Sigma_Q^\vee$, fix a minimal cone $\widetilde C\in \min(\Sigma_Q^\vee)$ that is contained in $C$.  We claim that $P$ is tightly contained in $C+v \in Y_{\mathscr V}$ if and only if $P$ is tightly contained in exactly one of $\{\widetilde C + \widetilde v \mid \widetilde v \in \mathscr V\textnormal{ and } \widetilde v \in \lineal(C)+v\} \subset X_{\mathscr V}$.  Assuming the claim for now, we have an equality of functions
\[
\overline{j_{C+v}} = \sum_{\widetilde v \in \mathscr V \cap (\lineal(C)+v)} \overline{j_{\widetilde C + \widetilde v}}
\]
on $\mathbb I(\Def(Q)|_{\mathscr V})$ for each $C+v\in Y_{\mathscr V}$.  Thus, since the map $\mathcal F_{\mathscr V}$ is defined by
\[
\mathcal F_{\mathscr V}(\one_P) =  \sum_{C+v\in X_{\mathscr V}} \tc{C+v}(P)\be_{C+v},
\]
the map $\mathcal E_{\mathscr V}: \ZZ^{X_{\mathscr V}} \to \ZZ^{Y_{\mathscr V}} \hookrightarrow \ZZ^{\oplus\tran(\Sigma_Q^\vee)}$ defined by
\[
\be_{C' + v'} \mapsto \sum_{\substack{C+v \in Y_{\mathscr V}\\ \text{ such that }\widetilde C = C' \text{ and} \\ v' \in \lineal(C) + v}} \be_{C+v}
\]
satisfies the commutativity of the diagram stated in the proposition.  The map $\mathcal E_{\mathscr V}$ is the dual map of the map $\mathcal E^*_{\mathscr V}: \ZZ^{Y_{\mathscr V}} \to \ZZ^{X_{\mathscr V}}$ defined by $\be_{C+v} \mapsto \sum_{\widetilde v\in \mathscr V \cap (\lineal(C)+v)} \be_{\widetilde C+ \widetilde v}$.  Since $\mathcal E^*_{\mathscr V}$ is identity when restricted to $\ZZ^{X_{\mathscr V}}$, and in particular surjective, the dual map  $\mathcal E_{\mathscr V}$ is injective.

For the claim, the ``if'' direction is immediate.  For the ``only if'' direction, suppose $C+v$ tightly contains $P$.  Let $y \in \operatorname{relint}(\widetilde C^\vee)$, and consider the face of $P$ maximizing the linear functional $\langle y, \cdot \rangle$.  Because $P$ is a deformation of $Q$, this face is necessarily a vertex $\widetilde v\in \mathscr V$ of $P$, and moreover the tangent cone of $P$ at $\widetilde v$ is contained in $\widetilde C$.  Hence, the translate $\widetilde C + \widetilde v$ tightly contains $P$, and $\widetilde v\in \lineal(C)+v$ since $C^\vee$ is a face of $\widetilde C^\vee$.
This translate of $\widetilde C$ is the unique one because any polyhedron is tightly contained in at most one translate of a cone.
\end{proof}

\begin{remark}\label{rem:restrictedFinvar}
In the proof of \Cref{prop:restrictedFinvar} above, let $\operatorname{St}(C) := \{C' \in \min(\Sigma_Q^\vee) \mid C' \subseteq C\}$ for each $C\in \Sigma_Q^\vee$.  Then, by the claim in the proof, for each $C+v\in Y_{\mathscr V}$ we have an equality of functions
\[
|\operatorname{St}(C)| \cdot \tc{C+v} = \sum_{C'\in \operatorname{St}(C)} \ \sum_{v' \in \mathscr V \cap (\lineal(C)+v)} \tc{C'+v'}.
\]
Thus, if we are working with $\QQ$-coefficients instead of $\ZZ$ in the previous proposition, we can alternatively define $\mathcal E_{\mathscr V}: \QQ^{X_{\mathscr V}} \to \QQ^{Y_{\mathscr V}}$ by
\[
\be_{C' + v'} \mapsto \sum_{\substack{C+v\in Y_{\mathscr V}\\ \text{ such that } C' \in \operatorname{St}(C) \text{ and}\\ v'\in \lineal(C) +v}} \frac{1}{|\operatorname{St}(C)|}\be_{C+v}.
\]
This eliminates making choices of the minimal cone $\widetilde C \in \operatorname{St}(C)$ for each cone $C\in \Sigma_Q^\vee$ in the proof of \Cref{prop:restrictedFinvar} when one works over $\QQ$-coefficients. 
\end{remark}

\section{Valuative invariants}\label{sec:valuativeInvariants}

We now consider universal valuative invariants.  In \S\ref{sec:val invariant of gen Coxeter permutohedra}, we compute the universal valuative invariant of extended generalized Coxeter permutohedra.  In \S\ref{sec:Coxeter matroids and the G-invariant}, we review Coxeter matroids and the $\mathcal G$-invariant, and discuss an application to delta-matroids. In \S\ref{sec:ProofGinvar}, we prove \Cref{mainthm:Ginvar}.

\subsection{Valuative invariants of generalized Coxeter permutohedra}\label{sec:val invariant of gen Coxeter permutohedra}

Let us first fix notations and recall basic facts about root systems.  See \cite{Hum90} for a general reference on reflection groups, and \cite{ACEP20} for an account tailored towards generalized Coxeter permutohedra.

\begin{notation}
As before, let $W$ be a finite reflection group with the root system $\Phi = (V,R)$.  We assume $\Phi$ to be reduced.  Let $n = \dim V$ and write $[n] = \{1, \ldots n\}$.  Let us fix once and for all a system of positive roots $R^+\subset R$, and set 
\begin{itemize}
\item $\{\alpha_1, \ldots, \alpha_n\}\subset R^+$ to be the set of simple roots of $\Phi$,
\item $\{s_1, \ldots, s_n\}\subset W$ to be the corresponding simple reflections generating $W$, and
\item $\{\varpi_1, \ldots, \varpi_n\} \subset V$ to be the corresponding set of fundamental weights.
\end{itemize}
For a subset $I\subset [n]$, we set
\begin{itemize}
\item $W_I= \langle s_i \mid i \in I\rangle$ to be the parabolic subgroup of $W$ corresponding to $I$,
\item $C_I := \operatorname{Cone}(\{\alpha_1, \ldots, \alpha_n\} \cup \{-\alpha_i \mid i \in I\})$, whose dual cone is
\item $\sigma_{[n]\setminus I} := \operatorname{Cone}(\varpi_i \mid i\in [n]\setminus I)$,
 and
\item $\varpi_{[n]\setminus I} := \sum_{i\in [n]\setminus I}\varpi_i$, whose $W$-orbit $W\cdot \varpi_{[n] \setminus I}$ is identified with $W/W_I$.
\end{itemize}
\end{notation}

The Coxeter complex $\Sigma_\Phi$ consists of the $W$-translates of the cones $\sigma_{I}$ as $I$ ranges over all subsets of $[n]$.  See \cite[\S3.2]{ACEP20} for a summary of some combinatorial properties of $\Sigma_\Phi$.  We will use the following standard fact about the action of $W$ on $\Sigma_\Phi$.

\begin{proposition}\label{prop:Worbits}
The $W$-orbits of the action of $W$ on $\Sigma_\Phi$ are in bijection with $2^{[n]}$, where a subset $I\subseteq [n]$ corresponds to the orbit $W\cdot \sigma_I$ with the stabilizer of $\sigma_I$ being $W_I$.  Similarly, the $W$-orbits of $\Sigma_\Phi^\vee$ are the orbits of $C_I$ as $I$ ranges over all subsets of $[n]$.
\end{proposition}

The reflection group $W$ acts on $V$, inducing an action of $W$ on $\ZZ^V$ by $(w\cdot f)(v) = f(w^{-1}v)$.  If $\mathscr P$ is a $W$-invariant family of polyhedra in $V$, the group $W$ thus acts on $\mathbb I(\mathscr P)$ by $w\cdot \one_P = \one_{w\cdot P}$.  We say that a valuative function $f: \mathscr P \to A$ is a \textbf{valuative invariant} if $f(w\cdot P) = f(P)$ for all $w\in W$ and $P \in \mathscr P$.  We will often use the following standard facts about action of a finite group on vector spaces; see for instance \cite[\S2.5.1]{AM10}.

\begin{lemma}\label{lem:averagemap}
Let a finite group $W$ act on a $\QQ$-vector space $U$.  The ``average map'' $\operatorname{avg}: U\to U$ defined by
\[
\operatorname{avg}(u) := \frac{1}{|W|} \sum_{w\in W} w\cdot u
\]
has the following properties:
\begin{enumerate}
\item It is a projection onto the space $U^W := \{u \in U \mid w \cdot u = u \text{ for all } w\in W\}$ of $W$-fixed points, which is identified with the space $U/_W := U / \operatorname{span}(u - w\cdot u\mid u\in U,\ w\in W)$ since $\ker \operatorname{avg} = \operatorname{span}(w\cdot u - u \mid u\in U,\ w\in W)$.
\item A map $f: U \to A$ to a $\QQ$-vector space $A$ satisfies $f(w\cdot u) = f(u)$ for all $w\in W$ and $u\in U$ if and only if $f$ factors as $f = f|_{U^W} \circ \operatorname{avg}$.  In other words, a $W$-invariant map from $U$ is equivalent to a map from $U^W$.
\end{enumerate}
\end{lemma}

We now compute the universal valuative invariant of $\mathsf{GP}_\Phi^+$ over $\QQ$-coefficients.  First, note that by   \Cref{prop:Worbits}, the orbits of the action of $W$ on $\tran(\Sigma_\Phi^\vee)$ are in bijection with $\{C_I + v \mid I \subseteq [n],\ v\in V\}$.  Since $C_I + v = C_{I'} + v'$ if and only if $I = I'$ and $v' - v \in \lineal(C_I) = \operatorname{span}(\alpha_i \mid i \in I)$, let us denote by $2^{[n]}\boxtimes V$ the set of equivalence classes of pairs $(I,v) \in 2^{[n]}\times V$ where $(I,v) \sim (I',v')$ if $I = I'$ and $v - v'\in \operatorname{span}(\alpha_i \mid i\in I)$.  The set $2^{[n]}\boxtimes V$ is in bijection with the $W$-orbits of $\tran(\Sigma_\Phi^\vee)$.

\begin{corollary}\label{cor:G+invar}
Write $\{U_{I,v}\}_{(I,v) \in 2^{[n]}\boxtimes V}$ for the standard basis of $\QQ^{\oplus (2^{[n]}\boxtimes V)}$.  The map
\[
\mathcal G^+: \mathsf{GP}_\Phi^+ \to \QQ^{\oplus (2^{[n]}\boxtimes V)}\quad\text{defined by}\quad P \mapsto \sum_{(I,v) \in 2^{[n]}\boxtimes V} \Big(\sum_{w\in W} \tc{w\cdot(C_I+v)}(P)\Big) U_{I,v}
\]
is the universal valuative invariant over $\QQ$.  That is, for any valuative invariant $g: \mathsf{GP}_\Phi^+ \to A$ to a $\QQ$-vector space $A$, there exists a unique linear map $\psi:  \QQ^{\oplus (2^{[n]}\boxtimes V)} \to A$ such that $\psi \circ \mathcal G^+ = g$.
\end{corollary}

\begin{proof}
Let us denote $\mathbb I(\mathsf{GP}_\Phi^+)_\QQ := \mathbb I(\mathsf{GP}_\Phi^+) \otimes \QQ$, and again abuse notation to write $\mathcal G^+$ for the function $\mathcal G^+: \mathbb I(\mathsf{GP}_\Phi^+)_\QQ \to \QQ^{\oplus (2^{[n]}\boxtimes V)}$ defined by $\one_P \mapsto \mathcal G^+(P)$.
We need to show that $\mathcal G^+$ is a $\QQ$-linear isomorphism.

Let $W$ act on $\ZZ^{\oplus \tran(\Sigma_\Phi^\vee)}$ by its action on $ \tran(\Sigma_\Phi^\vee)$.
Note first that the isomorphism $\mathcal F: \mathbb I(\mathsf{GP}_\Phi^+) \overset\sim\to \ZZ^{\oplus \tran(\Sigma_\Phi^\vee)}$ in \Cref{thm:tangent cone basis j} is $W$-equivariant:  For $w\in W$ and a translate of a tangent cone $C+v$, we have $\tc{C+v}(P) = 1$ if and only if $\tc{w\cdot(C+v)}(w\cdot P)=1$ for any $P\in \mathsf{GP}_\Phi^+$, and hence
\[
\sum_{C+v\in \tran{\Sigma_\Phi^\vee}} \tc{C+v}(w\cdot P)\be_{C+v} = \sum_{C+v\in \tran{\Sigma_\Phi^\vee}} \tc{C+v}(P)\be_{w\cdot(C+v)}.
\]
We thus have a commuting diagram
\[
\begin{tikzcd}
&\mathbb I(\mathsf{GP}_\Phi^+)_\QQ \arrow[r, "\mathcal F", "\simeq"'] \arrow["\operatorname{avg}"', d]&\QQ^{\tran(\Sigma_\Phi^\vee)} \arrow[d, "\operatorname{avg}"]\\
&\mathbb I(\mathsf{GP}_\Phi^+)_\QQ^W \arrow[r, "\simeq"] &(\QQ^{\tran(\Sigma_\Phi^\vee)})^W.
\end{tikzcd}
\]
Since the $W$-orbits of $\tran(\Sigma_\Phi^\vee)$ are in bijection with $2^{[n]}\boxtimes V$, we identify $(\QQ^{\tran(\Sigma_\Phi^\vee)})^W$ with $\QQ^{\oplus(2^{[n]}\boxtimes V)}$ by identifying $\operatorname{avg}(\be_{C_I+v})$ with $U_{I,v}$.
Then the composition $\operatorname{avg} \circ \mathcal F$ is the map $\mathcal G^+$, since
\[
\begin{split}
(\operatorname{avg}\circ \mathcal F)(P) & = \operatorname{avg} \Big( \sum_{C+v \in \tran(\Sigma_\Phi^\vee)} \tc{C+v}(P) \be_{C+v} \Big)\\
&= \operatorname{avg} \Big(\sum_{(I,v) \in 2^{[n]}\boxtimes V \ }  \sum_{w\in W}\tc{w\cdot(C_I+v)}(P) \be_{w\cdot(C_I +v)} \Big)\\
&= \sum_{(I,v)\in 2^{[n]}\boxtimes V \ }\sum_{w\in W} \tc{w\cdot(C_I+v)}(P) \operatorname{avg}(\be_{w\cdot(C_I +v)})\\
&= \sum_{(I,v)\in 2^{[n]}\boxtimes V} \Big( \sum_{w\in W} \tc{w\cdot(C_I+v)}(P) \Big) U_{I,v}
\end{split}
\]
\Cref{lem:averagemap} now implies that $\mathcal G^+$ is the universal valuative invariant.
\end{proof}

\subsection{Coxeter matroids, the \texorpdfstring{$\mathcal G$}{G}-invariant, and interlace polynomials}\label{sec:Coxeter matroids and the G-invariant}

We start with a brief treatment of Coxeter matroids. See \cite{BGW03} for a detailed account.

\medskip
We use $u \leq u'$ for the {Bruhat order} on the elements $u,u'$ of $W$, and write $u \leq^w u'$ to mean $w^{-1} u \leq w^{-1} u'$.  For a parabolic subgroup $W_I$ of $W$ corresponding to a subset $I\subseteq [n]$, the Bruhat order on $W/W_I$ is given by $B \leq^w B'$ for $B,B'\in W/W_I$ if $u \leq^w u'$ for some $u\in B$ and $u'\in B'$.  Recalling that the stabilizer of $\varpi_{[n]\setminus I} = \sum_{i\in [n]\setminus I}$ is $W_I$, for $B\in W/W_I$ we denote
\[
\delta_B := u \cdot \varpi_{[n]\setminus I} \quad\text{for any $u\in B$}.
\]
For a subset $S\subseteq W/W_I$, denote by $P_S$ the polytope
\[
P_S := \operatorname{Conv}(\delta_B \mid B \in S).
\]
The following generalization of \cite{GGMS87} establishes Coxeter matroids as subfamilies of generalized Coxeter permutohedra.

\begin{theorem}\cite[Theorem 6.3.1]{BGW03}\label{thm:GGMSCoxeter}
Let $S$ be a subset of $W/W_I$.  The following are equivalent:
\begin{enumerate}
\item The subset $S \subseteq W/W_I$ is a Coxeter matroid.  That is, for every $w\in W$ there exists a unique $\leq^w$-minimal element in $S$ (\Cref{def:Coxetermatroid}).
\item The polytope $P_S$ is a deformation of $\Pi_\Phi$, that is, $P_S \in \mathsf{GP}_\Phi^+$.
\end{enumerate}
\end{theorem}

For a Coxeter matroid $M \subseteq W/W_I$, we call the elements of $M$ the \textbf{bases} of $M$, and the $\leq^w$-minimal basis of $M$ is denoted $\min^w(M)$.  The polytope $P_M$ is called the \textbf{base polytope} of $M$.  Its vertices are $\{\delta_B \mid B\in M\}$.  \Cref{thm:GGMSCoxeter} states that Coxeter matroids of type $(\Phi,I)$ are exactly the polytopes in $\mathsf{GP}_\Phi^+$ whose vertices are subsets of $\mathscr V_{I} := \{\delta_B \mid B \in W/W_I\}$.
Hence, we consider
\[
\mathsf{Mat}_{\Phi,I} := \{\text{Coxeter matroids of type $(\Phi,I)$}\}
\]
as a subfamily of $\mathsf{GP}_\Phi^+$.  Recall that the \textbf{$\mathcal G$-invariant} of Coxeter matroids of type $(\Phi,I)$ is the function
\[
\mathcal G: \mathsf{Mat}_{\Phi,I} \to \QQ^{W/W_I} = \QQ\{U_B \mid B\in W/W_I\} \quad\text{defined by}\quad M \mapsto \sum_{w\in W} U_{w^{-1}  \min^w(M)}.
\]
\Cref{mainthm:Ginvar} states that $\mathcal G$ is the universal valuative invariant of $\mathsf{Mat}_{\Phi,I}$ over $\QQ$-coefficients.  For ordinary matroids, the specialization of the $\mathcal G$-invariant to the Tutte polynomial was computed in \cite{DF10}.  Here we feature an example of non-ordinary $(\Phi,I)$-matroids where the $\mathcal G$-invariant specializes to a well-studied polynomial invariant.

\medskip
Let $\Phi = B_n$ be the type $B$ root system, so that $W = S_n^B$, the signed permutation group.  Let $W_I$ be the parabolic subgroup such that $S_n^B/W_I = (\ZZ/2\ZZ)^n$.  In other words, we have that $\delta_{W_I}$ is the fundamental weight $\varpi_n = \frac{1}{2}(\be_1 + \cdots + \be_n)$, and that $\mathscr V_I = \{\frac{1}{2}(\pm \be_1 \pm \cdots \pm \be_n)\}$.  In this case, the $(\Phi,I)$-matroids are also known as \textbf{delta-matroids}, which were originally studied by Bouchet \cite{Bou89} and rediscovered as combinatorial abstractions of graphs embedded on surfaces \cite{CMNR19a, CMNR19b}.
A well-studied invariant of delta-matroids is the interlace polynomial.

\begin{definition}\label{defn:interlace}
Identifying the positive coordinates of an element in $\mathscr V_I = \{\frac{1}{2}(\pm \be_1 \pm \cdots \pm \be_n)\}$ with the subset of $[n] = \{1,\ldots, n\}$, consider a delta-matroid $M$ as a collection of subsets of $[n]$.  The \textbf{interlace polynomial} of $M$ is a univariate polynomial defined by
\[
q_M(x) := \sum_{A\subseteq [n]} x^{d_M(A)} \quad\text{where}\quad  d_M(A) := \min\{|(B-A) \cup (A-B)| \ \colon B\in M\} \text{ for } A\subseteq[n].
\]
\end{definition}

Interlace polynomials were originally defined in a study arising from DNA sequencing \cite{ABS00, ABS04}, and were generalized to delta-matroids in \cite{BH14}; see \cite{Mor17} for a survey.  Here we show that the interlace polynomial is a specialization of the $\mathcal G$-invariant in the following way.

\begin{theorem}\label{thm:interlacePolynomial}
Denote by $|\delta_B|$ the coordinate sum of the vector $\delta_B \in \mathscr V_I$.  For a delta-matroid $M$ we have
\[
q_M(x) = \frac{1}{n!} \sum_{w\in S_n^B} x^{-|\delta_{w^{-1}\min^w(M)}|+ \frac{n}{2}}.
\]
In particular, the interlace polynomial is a specialization of the $\mathcal G$-invariant, and hence is a valuative invariant of delta-matroids.
\end{theorem}

\begin{proof}
Again identifying the positive coordinates of an element in $\mathscr V_I = \{\frac{1}{2}(\pm \be_1 \pm \cdots \pm \be_n)\}$ with the subset of $[n] = \{1,\ldots, n\}$, let us write $\delta_A := \frac{1}{2}\left(\sum_{i\in A} \be_i - \sum_{i\notin A}\be_i \right)$ for a subset $A\subseteq [n]$.  We claim that if $w\in W$ satisfies $\delta_{wW_I} = \delta_A$, then
\[
d_M(A) = -|\delta_{w^{-1}\min^w(M)}| + \frac{n}{2}.
\]
Since $W_I$, which is the stabilizer of $\delta_{W_I}$, has order $n!$, the proposition then follows immediately from the claim and the definition of the interlace polynomial.

For the proof of our claim, first observe that for two subsets $A,B \subseteq [n]$, one has
\[
|(B-A) \cup (A-B)| = - \langle 2\delta_A, \delta_B\rangle + \frac{n}{2}.
\]
Thus, when $A$ is fixed, the elements $B\in M$ that achieve the minimum value $d_M(A)$ are exactly the ones that maximize the pairing $\langle \delta_A, \delta_B\rangle$.  Such $B\in M$ are exactly $\min^w(M)$ as $w\in W$ ranges over all elements of $W$ such that $\delta_{w\cdot W_I} = \delta_{A}$.  Thus, we have
\[
d_M(A) = -\langle 2 \delta_{wW_I}, \delta_{\min^w(M)}\rangle + \frac{n}{2} = -\langle 2 \delta_{W_I}, \delta_{w^{-1}\min^w(M)}\rangle + \frac{n}{2} = -|\delta_{w^{-1}\min^w(M)}| + \frac{n}{2},
\]
as claimed.
\end{proof}

\subsection{Proof of \texorpdfstring{\Cref{mainthm:Ginvar}}{Theorem B}}\label{sec:ProofGinvar}
Let us fix $I\subseteq [n]$, and recall the notation $\mathscr V_I = \{\delta_B \mid B\in W/W_I\}$.
We first recall a standard fact relating the positions of the points $\mathscr V_I$ to the Bruhat order on $W/W_I$.

\begin{lemma}\label{lem: difference positive root}\cite[Lemma 6.2.4]{BGW03}
Let $B,B'\in W/W_I$ be such that $\delta_{B'} - \delta_{B} = \lambda \alpha$ for some root $\alpha \in R$ and $\lambda\geq 0$.    Then for $w\in W$,  one has $B\leq^w B'$ if and only if $\alpha \in w \cdot R^+$.  In particular, if $B\leq^w B'$, then $\delta_{B'} \in w\cdot C_\emptyset + \delta_B$, where $w\cdot C_\emptyset$ is the $w$-permutation of the positive root cone $C_\emptyset = \operatorname{Cone}(\alpha_1, \ldots, \alpha_n)$.
\end{lemma}

\Cref{lem: difference positive root} allows us to express the $\mathcal G$-invariant in terms of tight containments in the following way, thereby showing that it is a valuative invariant.

\begin{proposition}\label{prop:GinvarByTC}
For any Coxeter matroid $M \in \mathsf{Mat}_{\Phi,I}$, we have an equality
\[
\mathcal G (M) = \sum_{B\in W/W_I} \Big(\sum_{w\in W} \tc{w\cdot (C_\emptyset + \delta_B)}(P_M) \Big)U_B.
\]
In particular, the $\mathcal G$-invariant is valuative, and is an invariant.
\end{proposition}

\begin{proof}
For a fixed $w\in W$, the base polytope $P_M$ is tightly contained in at most one translate $w\cdot C_\emptyset + \delta_B$ of the strongly convex cone $w \cdot C_\emptyset$.  Since $\min^w(M)$ is the unique minimal element in $M$ with respect to $\leq^w$, \Cref{lem: difference positive root} implies that $w\cdot C_\emptyset + \delta_{\min^w(M)} = w\cdot(C_\emptyset + \delta_{w^{-1}\cdot\min^w(M)})$ tightly contains $P_M$.  We now compute that
\[
\begin{split}
\sum_{B\in W/W_I} \Big(\sum_{w\in W} \tc{w\cdot (C_\emptyset + \delta_B)}(P_M) \Big)U_B &= \sum_{w\in W} \sum_{B\in W/W_I} \tc{w\cdot (C_\emptyset + \delta_B)}(P_M)U_B\\
&= \sum_{w\in W} U_{w^{-1}\cdot\min^w(M)},
\end{split}
\]
as desired.  The function $\mathcal G$ is valuative since the functions $\tc{w\cdot(C_\emptyset + \delta_B)}$ are valuative on $\mathsf{GP}_\Phi^+$ by \Cref{thm: tight containment is val}, and hence valuative on the subfamily $\mathsf{Mat}_{\Phi,I}$.  The function $\mathcal G$ is evidently $W$-invariant.
\end{proof}

We can now prove the first half of \Cref{mainthm:Ginvar}.

\begin{proposition}\label{prop: existence}
For any valuative invariant $g: \mathsf{Mat}_{\Phi,I} \to A$ to a $\QQ$-vector space $A$, there exists a map $\psi: \QQ^{W/W_I} \to A$ such that $g= \psi\circ \mathcal G$.
\end{proposition}

We will prove that the map $\psi$ is unique in \Cref{prop: uniqueness}, thereby finishing the proof of \Cref{mainthm:Ginvar}.

\begin{proof}
Using the notation as in the discussion preceding \Cref{prop:restrictedFinvar}, we have $\mathsf{Mat}_{\Phi,I} = \Def(\Pi_\Phi)|_{\mathscr V_I}$, where $\mathscr V_I = \{\delta_B \mid B\in W/W_I\}$.  Moreover, we have $X_{\mathscr V_I} = \{w \cdot C_\emptyset + \delta_B \mid w\in W,\ B\in W/W_I\}$ since the minimal cones of $\Sigma_\Phi^\vee$ are $W$-permutations of the positive root cone $C_\emptyset$.  Considering the action of $W$ on $\QQ^{X_{\mathscr V_I}}$, we can identify $(\QQ^{X_{\mathscr V_I}})^W$ with $\QQ^{W/W_I}$ by identifying $\operatorname{avg}(\be_{C_\emptyset + \delta_B})$ with $U_B$.  Then, the map $\mathcal G$ is the composition $\operatorname{avg} \circ \mathcal F_{\mathscr V_I}$, since
\[
\begin{split}
(\operatorname{avg} \circ \mathcal F_{\mathscr V_I})(P_M) &= \operatorname{avg} \Big( \sum_{B\in W/W_I \ }\sum_{w\in W} \tc{w\cdot (C_\emptyset + \delta_B)}(P_M)\be_{w\cdot(C_\emptyset+\delta_B)}  \Big)\\
 &= \sum_{B\in W/W_I} \Big(\sum_{w\in W} \tc{w\cdot (C_\emptyset + \delta_B)}(P_M) \Big)U_B \\
&=\mathcal G(M),
\end{split}
\]
where the last equality follows from \Cref{prop:GinvarByTC}.
Now, \Cref{prop:restrictedFinvar} states that there is an injection $\mathcal E_{\mathscr V_I}: \QQ^{X_I} \to \QQ^{\oplus \tran(\Sigma_\Phi^\vee)}$ fitting into the left square of the commuting diagram
\[
\begin{tikzcd}
\mathbb I(\mathsf{Mat}_{\Phi,I})_\QQ \arrow[r, "\mathcal F_{\mathscr V_I}"] \arrow[d, hook] & \QQ^{X_{\mathscr V_I}} \arrow[d, hook, "\mathcal E_{\mathscr V_I}"] \arrow[r, "\operatorname{avg}"] &\QQ^{W/W_I} \arrow[d, hook]\\
\mathbb I(\mathsf{GP}_\Phi^+)_\QQ \arrow[r, "\mathcal F"] & \QQ^{\oplus\tran(\Sigma_Q^\vee)} \arrow[r, "\operatorname{avg}"] &\QQ^{2^{[n]}\boxtimes V}.
\end{tikzcd}
\]
Moreover, \Cref{rem:restrictedFinvar} implies that over $\QQ$-coefficients the map $\mathcal E_{\mathscr V_I}$ can be made $W$-equivariant, so that the right square of the diagram above also commutes.  Extend the valuative invariant $g$ on $\mathbb I(\mathsf{Mat}_{\Phi,I})_\QQ$ to any valuative invariant on $\mathbb I(\mathsf{GP}_\Phi^+)_\QQ$, for example, by extending $g$ to a function on $\mathbb I(\mathsf{GP}_\Phi^+)_\QQ$ (which may not be $W$-invariant) and then precomposing with the averaging map.  Then, since $\mathcal G = \operatorname{avg} \circ \mathcal F_{\mathscr V_I}$, the universality of $\operatorname{avg} \circ \mathcal F = \mathcal G^+$ from \Cref{cor:G+invar} implies that there exists $\psi$ such that $g = \psi \circ \mathcal G$.
\end{proof}

We remark that the map $\mathcal{F}_{\mathscr V_I}$ in the proof above is generally not an isomorphism.  To prove the other half of \Cref{mainthm:Ginvar},
we first consider the following family of Coxeter matroids, which are special cases of Bruhat interval polytopes studied in \cite{TW15} and \cite{CDM20}.

\begin{proposition}\cite[Theorem 4.5]{CDM20}
For $B\in W/W_I$, the subset
\[
\Omega_{B} := \{B'\in W/W_I \mid B \leq B'\}
\]
is a Coxeter matroid.
\end{proposition}

\begin{definition}\label{def:schubert}
We call the Coxeter matroid $\Omega_{B}$ in the proposition above the \textbf{Coxeter Schubert matroid} with respect to $B \in W/W_I$.
\end{definition}

\begin{remark}\label{rem:geomrem1}
Suppose $\Phi$ is crystallographic.
Let $G$ be the associated Lie group with a chosen Borel subgroup arising from our fixed choice of positive system for $\Phi$.
Let $P_I$ be the parabolic subgroup of $G$ corresponding to the subset $I$ of simple roots, and $T$ the torus in $G$.
Coxeter Schubert matroids correspond to the Bruhat cells of the flag variety $G/P_I$ in the following way.
For $B\in W/W_I$, consider the torus-orbit closure $\overline{T\cdot x}$ of a general point $x$ in the Bruhat cell of $G/P_I$ corresponding to $B$.
The authors of \cite{TW15} show that the base polytope of $\Omega_B$ is the moment polytope of $\overline{T\cdot x}$, and hence a Coxeter matroid.
\end{remark}

The following key lemma establishes an ``upper-triangularity'' property of Coxeter Schubert matroids that we will need for establishing uniqueness of the map $\psi$ in \Cref{prop: uniqueness}.

\begin{lemma}\label{lem: upper triangle} Let $B\in W/W_I$.  There exist scalars $c^B_{B'} \in \QQ$ for $B'\leq B \in W/W_I$ such that
 \[ \mathcal{G}(\Omega_B) = \sum_{B' \leq B} c^B_{B'}U_{B'}, \quad\text{where}\quad c_B^B \not = 0.\]
\end{lemma}

\begin{proof}[Proof of \Cref{lem: upper triangle}]
We claim that for $B \in W/W_I$ and $w \in W$, we have
 \[ w^{-1} \cdot \min ^{w}(\Omega_B) \leq B.\]
The lemma then follows immediately from the claim since
\[ \mathcal{G}(\Omega_B) = \sum_{w \in W} U_{w^{-1} \cdot \min^w(\Omega_B)},\]
and the coefficient $c_{B}^B \not = 0$ because $w^{-1} \cdot \min^w(\Omega_B) = B$ when $w = e$ the identity.
 
We now prove the claim by borrowing tools from 0-Hecke algebras.  Let us first recall some basic properties; see \cite{Nor79} for general reference.  The $0$-Hecke algebra $\mathcal{H}_0$ of $W$ is the unital $\mathbb Q$ algebra generated by $\{T_i\}_{i \in [n]}$, subject to the relations:

\begin{itemize}
    \item $T_i^2 = -T_i$ for all simple reflections $s_i$ ($i = 1, \ldots, n$), and 
    \item $(T_iT_jT_i \cdots)_{n_{ij}} = (T_jT_iT_j \cdots )_{n_{ij}}$ where $n_{ij}$ is the order of $s_is_j$ in $W$ and $(T_iT_jT_i \cdots)_{n_{ij}}$ is the product of the first $n_{ij}$ terms in the sequence $T_i, T_j, T_i,\ldots$.
\end{itemize}

For any $w = s_{i_1} \cdots s_{i_k} \in W$, let $T_w$ denote the product $T_{i_1} \cdots T_{i_k}$. One can show that this product is the same for any reduced expression of $w$. The only result we use is the following \cite[1.3]{Nor79}:
\begin{equation}\label{eq:Matosumo}
T_w \cdot T_{s_i} = \left \{
    \begin{array}{cc}
    T_{ws_i} & \text{if $\ell(ws_i) > \ell(w)$} \\
    -T_{w} & \text{else,}
    \end{array} \right. 
\qquad\text{and}\qquad
T_{s_i} \cdot T_w = \left \{
    \begin{array}{cc}
        T_{s_iw} & \text{if $\ell(s_iw) > \ell(w)$} \\
        -T_{w} & \text{else.}
    \end{array} \right .
\end{equation}
From this, it follows that for $u,v,w \in W$ one has
\begin{equation}\label{eq:conseq}
T_w \cdot T_u = \pm T_v \implies v\geq u \text{ and } v\geq w.
\end{equation}

For the claim, it suffices to prove the result for $W_I = \{e\}$ since we have a projection map from $W$ to $W/W_I$.  Fix $b\in W$.  Since $w^{-1} \cdot \min^w(\Omega_b) = \min(w^{-1} \cdot \Omega_b),$ we need to show that for all $w \in W$, we have 
\[\min(w^{-1} \cdot \Omega_b) \leq b.\]
To show this, we find an element $c \in \Omega_b$ and an element $d \leq b$ such that $c = w \cdot d$, as $d$ will then satisfy $\min(w^{-1} \cdot \Omega_b) \leq d \leq b$.
Using (\ref{eq:Matosumo}), let $c \in W$ be the unique element where $T_w \cdot T_b = \pm T_c$. By property \eqref{eq:conseq}, we have that $c \geq b$ and hence it is in $\Omega_b$. Say that $w = s_{j_1} \cdots s_{j_k}$ and $b = s_{i_1} \cdots s_{i_\ell}$ are reduced expressions, then applying (\ref{eq:Matosumo}) to $T_w \cdot T_b = T_w \cdot T_{s_{i_1}} \cdots T_{s_{i_\ell}}$, we obtain a reduced expression of $c$ as
 \[ c = (s_{j_1} \cdots s_{j_k})b'\]
where $b'$ is a subword of $s_{i_1} \cdots s_{i_{\ell}} = b$. Setting $d = b'$, we see that $d \leq b$ and that $c = w \cdot d$.
\end{proof}

\begin{proposition}\label{prop: uniqueness}
Let $g: \mathbb I(\mathsf{Mat}_{\Phi,I})_\QQ \to A$ be a valuative invariant into a $\QQ$-vector space $A$. There exists at most one linear map $\psi: \QQ^{W/W_I} \to A$ such that
 \[ \psi \circ \mathcal{G}(M) = g(M).\]
\end{proposition}

\begin{proof} Suppose $\psi$ is a map satisfying the conditions of the proposition. We will show that $\psi$ is determined by the values on the Coxeter Schubert matroids by induction over the Bruhat order. For the base case, let $B = eW_I \in W/W_I$. Then the the equation
 \[ \psi \circ \mathcal{G}(\Omega_{B}) = c_{B}^{B} \psi(U_{B}) = g(\Omega_B)\]
implies that $\psi(U_{B}) = \frac{g(\Omega_B)}{c_{B}^{B}}$.
For any other $B \in W/W_I$, suppose we have already computed $\psi(U_{B'})$ for $B' \leq B$, then the equation
 \[\psi \circ \mathcal{G}(\Omega_B) = g(\Omega_B) =  c_B^B \psi(U_B) + \sum_{B' \lneq B} c_{B'}^{B} \psi(U_B') \]
gives the value of $\psi(U_B)$.
As the set $\{U_B\}$ is a basis for $\QQ^{W/W_I}$, this construction determines $\psi$. Hence, there is at most one such map.
\end{proof}

Note that this proposition also gives an algorithm to compute the specialization map $\psi$. Combining Proposition \ref{prop: existence} and Proposition \ref{prop: uniqueness} completes the proof of \Cref{mainthm:Ginvar}.

\begin{remark}\label{rem:badintersection}
Using \Cref{lem: upper triangle} can be avoided if the pair $(\Phi, I)$ satisfies the following property:  For any $v\in \mathscr V_I$, the intersection of $C_\emptyset + v$ with the polytope $P_{W/W_I} = \operatorname{Conv}(\delta_B \mid B\in W/W_I)$ is a base polytope of a $(\Phi, I)$-matroid.  Let us say that a pair $(\Phi, I)$ is \textbf{intersection-stable} if it satisfies this property.  If $A_n$ is the type $A$ root system of dimension $n$, the pair $(A_n, [n]\setminus i)$ is intersection-stable for each $i \in [n]$.  This feature of ordinary matroids is important in the argument of \cite{DF10} establishing $\mathcal G$ as the universal invariant of ordinary matroids.  However, intersection-stability is a rare property for general Coxeter systems:
 \begin{itemize}
     \item Even in type $A$, the pair $(A_3, \emptyset)$ is not intersection-stable.
     \item For $B_3$ (or $C_3$), let $\varpi_{1}$ and $\varpi_2$ be the fundamental weights such that their weight polytopes are not the cube. Then the pairs $(B_3, [3]\setminus 1)$ and $(B_3, [3]\setminus 2)$ are not intersection-stable (similarly for $C_3$).
     \item In type $D_4$, three of the four fundamental weights make the pair $(D_4, [4]\setminus i)$ not intersection-stable.
 \end{itemize}
The first point reflects that non-minuscule types usually fail to be intersection-stable since the converse of \Cref{lem: difference positive root} holds only for minuscule types \cite[4.1]{Pro84}.  The latter two points show that intersection-stability can fail even in minuscule types unless $\Phi = A_n$.  If $\varpi_n$ is the fundamental weight of $B_n$ so that the weight polytope is the hypercube, one can show that the pair $(B_n, [n]\setminus n)$ is intersection-stable.
\end{remark}

Two $(\Phi,I)$-matroids $M,M'$ are said to be  \textbf{isomorphic} if $M = w \cdot M'$ for some $w\in W$.  As an application of \Cref{mainthm:Ginvar}, we show that Coxeter Schubert matroids form a basis for the space
\[
\mathbb I(\mathsf{Mat}_{\Phi,I})/_W := \mathbb I(\mathsf{Mat}_{\Phi,I})_\QQ / \operatorname{span}(\one_M - \one_{w\cdot M} \mid M \in \mathsf{Mat}_{\Phi,I},\ w\in W)\]
of isomorphism classes of indicator functions of $(\Phi,I)$-matroids.

\begin{corollary}\label{cor:Schubertsbasis} For a Coxeter matroid $M$,  let $[M]$ denote the image of $\one_{P_M}$ in $\mathbb I(\mathsf{Mat}_{\Phi,I})/_W$.  Then the set $\{[\Omega_B]\}_{B \in W/W_I}$ of the classes of Coxeter Schubert matroids is a basis for $\mathbb I(\mathsf{Mat}_{\Phi,I})/_W$.
\end{corollary}

\begin{samepage}
\begin{proof}
Combined with \Cref{lem:averagemap}.(2), \Cref{mainthm:Ginvar} states that the map $\mathcal G: \mathbb I(\mathsf{Mat}_{\Phi,I})_\QQ^W \to \QQ^{W/W_I}$ is an isomorphism, and we have $\mathbb I(\mathsf{Mat}_{\Phi,I})/_W\simeq \mathbb I(\mathsf{Mat}_{\Phi,I})_\QQ^W$ by  \Cref{lem:averagemap}.(1).  Now, the ``upper-triangularity'' Lemma \ref{lem: upper triangle} states that we have
 \[ \mathcal{G}(\Omega_B) =  \sum_{B' \leq B} c^B_{B'}U_{B'} \quad \text{with} \quad c_B^B \neq 0,\]
which shows that $\{\mathcal G(\Omega_B)\}$ and $\{U_B\}$ are two bases of $\QQ^{W/W_I}$ related by an upper-triangular matrix.
\end{proof}
\end{samepage}

\begin{remark}
Suppose $\Phi$ is crystallographic, and let $G/P_I$ be the flag variety as in \Cref{rem:geomrem1}.  As Coxeter Schubert matroids correspond to Bruhat cells of $G/P_I$, we note the parallelism between \Cref{cor:Schubertsbasis} and the fact that (closures of) Bruhat cells of $G/P_I$ form a basis for the cohomology ring of $G/P_I$.  
\end{remark}

\subsection*{Acknowledgements}
We thank Alex Postnikov for helpful discussions which led to a sketch of the proof of \Cref{lem: upper triangle}.
We are also grateful to Federico Ardila for his helpful feedback, and to the Max Planck Institute for Mathematics in the Sciences for its hospitality.
C.\ Eur was supported by US National Science Foundation DMS-2001854 and UC Berkeley Dissertation Fellowship.
M.\ Sanchez was supported by the NSF Graduate Research Fellowship DGE 1752814.
M.\ Supina was supported by the Graduate Fellowships for STEM Diversity and an Erasmus+ Mobility grant.


\bibliographystyle{amsalpha}
\bibliography{sources}

\providecommand{\bysame}{\leavevmode\hbox to3em{\hrulefill}\thinspace}
\providecommand{\MR}{\relax\ifhmode\unskip\space\fi MR }
\providecommand{\MRhref}[2]{%
  \href{http://www.ams.org/mathscinet-getitem?mr=#1}{#2}
}
\providecommand{\href}[2]{#2}
\begin{thebibliography}{CMNR19b}

\bibitem[AA17]{AA17}
Marcelo {Aguiar} and Federico {Ardila}, \emph{{Hopf monoids and generalized
  permutahedra}}, arXiv e-prints (2017), arXiv:1709.07504.

\bibitem[ABS00]{ABS00}
Richard Arratia, B\'{e}la Bollob\'{a}s, and Gregory~B. Sorkin, \emph{The
  interlace polynomial: a new graph polynomial}, Proceedings of the {E}leventh
  {A}nnual {ACM}-{SIAM} {S}ymposium on {D}iscrete {A}lgorithms ({S}an
  {F}rancisco, {CA}, 2000), ACM, New York, 2000, pp.~237--245. \MR{1754863}

\bibitem[ABS04]{ABS04}
\bysame, \emph{The interlace polynomial of a graph}, J. Combin. Theory Ser. B
  \textbf{92} (2004), no.~2, 199--233. \MR{2099142}

\bibitem[ACEP20]{ACEP20}
Federico Ardila, Federico Castillo, Christopher Eur, and Alexander Postnikov,
  \emph{Coxeter submodular functions and deformations of {C}oxeter
  permutahedra}, Adv. Math. \textbf{365} (2020), 107039. \MR{4064768}

\bibitem[AFR10]{AFR10}
Federico Ardila, Alex Fink, and Felipe Rinc\'{o}n, \emph{Valuations for matroid
  polytope subdivisions}, Canad. J. Math. \textbf{62} (2010), no.~6,
  1228--1245. \MR{2760656}

\bibitem[AM10]{AM10}
Marcelo Aguiar and Swapneel Mahajan, \emph{Monoidal functors, species and
  {H}opf algebras}, CRM Monograph Series, vol.~29, American Mathematical
  Society, Providence, RI, 2010, With forewords by Kenneth Brown and Stephen
  Chase and Andr\'{e} Joyal. \MR{2724388}

\bibitem[AS20]{AS20}
Federico {Ardila} and Mario {Sanchez}, \emph{{Valuations and the Hopf Monoid of
  Generalized Permutahedra}}, arXiv e-prints (2020), arXiv:2010.11178.

\bibitem[BEZ20]{BEZ20}
Madeline {Brandt}, Christopher {Eur}, and Leon {Zhang}, \emph{{Tropical flag
  varieties}}, arXiv e-prints (2020), arXiv:2005.13727.

\bibitem[BGW03]{BGW03}
Alexandre~V. Borovik, I.~M. Gelfand, and Neil White, \emph{Coxeter matroids},
  Progress in Mathematics, vol. 216, Birkh\"{a}user Boston, Inc., Boston, MA,
  2003. \MR{1989953}

\bibitem[BH14]{BH14}
Robert Brijder and Hendrik~Jan Hoogeboom, \emph{Interlace polynomials for
  multimatroids and delta-matroids}, European J. Combin. \textbf{40} (2014),
  142--167. \MR{3191496}

\bibitem[Bou89]{Bou89}
Andr\'{e} Bouchet, \emph{Maps and {$\triangle$}-matroids}, Discrete Math.
  \textbf{78} (1989), no.~1-2, 59--71. \MR{1020647}

\bibitem[{Bri}37]{Bri37}
C.J. {Brianchon}, \emph{Th\'{e}or\`{e}me nouveau sur les poly\`{e}dres
  convexes}, J. \'{E}cole Polytechnique \textbf{15} (1837), 317--319.

\bibitem[CDM20]{CDM20}
Fabrizio Caselli, Michele D’Adderio, and Mario Marietti, \emph{Weak
  generalized lifting property, bruhat intervals, and coxeter matroids},
  International Mathematics Research Notices (2020), rnaa124.

\bibitem[CMNR19a]{CMNR19a}
Carolyn Chun, Iain Moffatt, Steven~D. Noble, and Ralf Rueckriemen,
  \emph{Matroids, delta-matroids and embedded graphs}, J. Combin. Theory Ser. A
  \textbf{167} (2019), 7--59. \MR{3938888}

\bibitem[CMNR19b]{CMNR19b}
\bysame, \emph{On the interplay between embedded graphs and delta-matroids},
  Proc. Lond. Math. Soc. (3) \textbf{118} (2019), no.~3, 675--700. \MR{3932785}

\bibitem[DF10]{DF10}
Harm Derksen and Alex Fink, \emph{Valuative invariants for polymatroids}, Adv.
  Math. \textbf{225} (2010), no.~4, 1840--1892. \MR{2680193}

\bibitem[FS12]{FS12}
Alex Fink and David~E. Speyer, \emph{{$K$}-classes for matroids and equivariant
  localization}, Duke Math. J. \textbf{161} (2012), no.~14, 2699--2723.
  \MR{2993138}

\bibitem[GGMS87]{GGMS87}
I.~M. Gelfand, R.~M. Goresky, R.~D. MacPherson, and V.~V. Serganova,
  \emph{{Combinatorial geometries, convex polyhedra, and Schubert cells}}, Adv.
  in Math. \textbf{63} (1987), no.~3, 301--316.

\bibitem[{Gra}74]{Gra74}
J.P. {Gram}, \emph{Om rumvinklerne i et polyeder}, Tidsskrift for Math
  \textbf{4} (1874), 161--163.

\bibitem[GS87a]{GS87a}
I.~M. Gelfand and V.~V. Serganova, \emph{Combinatorial geometries and the
  strata of a torus on homogeneous compact manifolds}, Uspekhi Mat. Nauk
  \textbf{42} (1987), no.~2(254), 107--134, 287. \MR{898623}

\bibitem[GS87b]{GS87b}
\bysame, \emph{On the general definition of a matroid and a greedoid}, Dokl.
  Akad. Nauk SSSR \textbf{292} (1987), no.~1, 15--20. \MR{871945}

\bibitem[HLT11]{HLT11}
Christophe Hohlweg, Carsten E. M.~C. Lange, and Hugh Thomas, \emph{Permutahedra
  and generalized associahedra}, Adv. Math. \textbf{226} (2011), no.~1,
  608--640. \MR{2735770}

\bibitem[Hum90]{Hum90}
James~E. Humphreys, \emph{Reflection groups and {C}oxeter groups}, Cambridge
  Studies in Advanced Mathematics, vol.~29, Cambridge University Press,
  Cambridge, 1990. \MR{1066460}

\bibitem[Joc19]{Joc19}
Katharina Jochemko, \emph{A brief introduction to valuations on lattice
  polytopes}, Algebraic and geometric combinatorics on lattice polytopes, World
  Sci. Publ., Hackensack, NJ, 2019, pp.~38--55. \MR{3971684}

\bibitem[Kap93]{Kap93}
M.~M. Kapranov, \emph{Chow quotients of {G}rassmannians. {I}}, I. {M}.
  {G}elfand {S}eminar, Adv. Soviet Math., vol.~16, Amer. Math. Soc.,
  Providence, RI, 1993, pp.~29--110. \MR{1237834}

\bibitem[Laf99]{Laf99}
L.~Lafforgue, \emph{Pavages des simplexes, sch\'{e}mas de graphes recoll\'{e}s
  et compactification des {${\rm PGL}^{n+1}_r/{\rm PGL}_r$}}, Invent. Math.
  \textbf{136} (1999), no.~1, 233--271. \MR{1681089}

\bibitem[Laf03]{Laf03}
\bysame, \emph{Chirurgie des grassmanniennes}, CRM Monograph Series, vol.~19,
  American Mathematical Society, Providence, RI, 2003. \MR{1976905}

\bibitem[McM93]{McM93}
Peter McMullen, \emph{Valuations and dissections}, Handbook of convex geometry,
  {V}ol. {A}, {B}, North-Holland, Amsterdam, 1993, pp.~933--988. \MR{1243000}

\bibitem[McM09]{McM09}
\bysame, \emph{Valuations on lattice polytopes}, Adv. Math. \textbf{220}
  (2009), no.~1, 303--323. \MR{2462842}

\bibitem[Mor17]{Mor17}
Ada Morse, \emph{The interlace polynomial}, Graph polynomials, Discrete Math.
  Appl. (Boca Raton), CRC Press, Boca Raton, FL, 2017, pp.~1--23. \MR{3790909}

\bibitem[Nak35]{Nak35}
Takeo Nakasawa, \emph{Zur axiomatik der linearen abhängigkeit. i}, Science
  Reports of the Tokyo Bunrika Daigaku, Section A \textbf{2} (1935), no.~43,
  235--255.

\bibitem[Nor79]{Nor79}
P.~N. Norton, \emph{0-hecke algebras}, Journal of the Australian Mathematical
  Society \textbf{27} (1979), no.~3, 337–357.

\bibitem[Pos09]{Pos09}
Alexander Postnikov, \emph{Permutohedra, associahedra, and beyond}, Int. Math.
  Res. Not. IMRN (2009), no.~6, 1026--1106. \MR{2487491}

\bibitem[Pro84]{Pro84}
Robert~A. Proctor, \emph{Bruhat lattices, plane partition generating functions,
  and minuscule representations}, Eur. J. Comb. \textbf{5} (1984), 331--350.

\bibitem[Rei92]{Rei92}
Victor Reiner, \emph{Quotients of {C}oxeter complexes and {$P$}-partitions},
  Mem. Amer. Math. Soc. \textbf{95} (1992), no.~460, vi+134. \MR{1101971}

\bibitem[Spe08]{Spe08}
David~E. Speyer, \emph{{Tropical linear spaces}}, SIAM J. Discrete Math.
  \textbf{22} (2008), no.~4, 1527--1558.

\bibitem[Spe09]{Spe09}
\bysame, \emph{A matroid invariant via the {$K$}-theory of the {G}rassmannian},
  Adv. Math. \textbf{221} (2009), no.~3, 882--913. \MR{2511042}

\bibitem[TW15]{TW15}
E.~Tsukerman and L.~Williams, \emph{Bruhat interval polytopes}, Adv. Math.
  \textbf{285} (2015), 766--810. \MR{3406515}

\bibitem[Whi35]{Whi35}
Hassler Whitney, \emph{On the {A}bstract {P}roperties of {L}inear
  {D}ependence}, Amer. J. Math. \textbf{57} (1935), no.~3, 509--533.
  \MR{1507091}

\end{thebibliography}

\end{document}